\documentclass[12pt]{amsart}
\usepackage{etex}
\usepackage{amsfonts, amssymb, latexsym}

\usepackage{amscd,amssymb}
\usepackage{verbatim}
\usepackage{pstricks}
\usepackage{pst-node}
\usepackage[mathscr]{euscript}
\usepackage{tikz}
\usetikzlibrary{matrix,arrows}
\usepackage[normalem]{ulem}
\usepackage{varwidth}
\usepackage{leftidx}

\newsymbol\pp 1275
\usepackage{tikz}
\usetikzlibrary{matrix,arrows,backgrounds,shapes.misc,shapes.geometric,patterns,calc,positioning}

\usepackage[colorlinks=true, pdfstartview=FitV, linkcolor=blue, citecolor=blue, urlcolor=blue]{hyperref}

\usepackage{fullpage}
\usepackage{graphicx}
\usepackage{enumerate}

\def\VR{\kern-\arraycolsep\strut\vrule &\kern-\arraycolsep}
\def\vr{\kern-\arraycolsep & \kern-\arraycolsep}

\newtheorem{theorem}{Theorem}

\newtheorem{lemma}[theorem]{Lemma}
\newtheorem{prop}[theorem]{Proposition}
\newtheorem{corollary}[theorem]{Corollary}

\theoremstyle{definition}
\newtheorem{definition}[theorem]{Definition}

\newtheorem{rmk}[theorem]{Remark}
\newenvironment{remark}[1][]{\begin{rmk}[#1]\pushQED{\qed}}{\popQED \end{rmk}}

\newtheorem{qu}[theorem]{Question}

\newtheorem*{rmknonum}{Remark}

\newtheorem{obs}[theorem]{Observation}

\newtheorem{ex}[theorem]{Example}

\newcommand{\End}{\operatorname{End}}

\newcommand{\tr}{\operatorname{Tr}}

\newcommand{\rep}{\operatorname{rep}}

\newcommand{\bl}{\operatorname{BL}}
\newcommand{\BL}{\mathcal{BL}}

\newcommand{\SI}{\operatorname{SI}}

\newcommand{\GL}{\operatorname{GL}}

\newcommand{\ZZ}{\mathbb Z}
\newcommand{\CC}{\mathbb C}
\newcommand{\RR}{\mathbb R}
\newcommand{\NN}{\mathbb N}
\newcommand{\QQ}{\mathbb Q}

\newcommand{\HH}{\mathbb H}

\newcommand{\tup}{\mathbf p}
\newcommand{\0}{\mathbf 0}

\newcommand{\V}{V}

\newcommand{\Id}{\mathbf{I}}

\newcommand{\Mat}{\operatorname{Mat}}

\newcommand{\Stab}{\operatorname{Stab}}

\newcommand{\ddim}{\operatorname{\mathbf{dim}}}

\newcommand{\dd}{\operatorname{\mathbf{d}}}

\newcommand{\G}{\mathcal{G}}

\newcommand{\ar}{\mathcal{A}}
\newcommand{\s}{\mathcal{S}}
\newcommand{\jj}{\mathcal{I}^{-}_j}
\newcommand{\ii}{\mathcal{I}^{+}_i}
\newcommand{\Span}{\mathsf{Span}}
\newcommand{\capa}{\mathbf{D}}
\newcommand{\Det}{\mathsf{Det}}

\newcommand{\rk}{\operatorname{rank}}

\newcommand\restr[2]{{
  \left.\kern-\nulldelimiterspace 
  #1 
  \vphantom{\big|} 
  \right|_{#2} 
  }}

\begin{document}
\title{The capacity of quiver representations and Brascamp-Lieb constants}
\author{Calin Chindris}
\address{University of Missouri-Columbia, Mathematics Department, Columbia, MO, USA}
\email[Calin Chindris]{chindrisc@missouri.edu}

\author{Harm Derksen}
\address{Northeastern University, Boston, MA}
\email[Harm Derksen]{ha.derksen@northeastern.edu}

\date{\today}
\bibliographystyle{amsalpha}

\begin{abstract} Let $Q$ be a bipartite quiver, $\V$ a real representation of $Q$, and $\sigma$ an integral weight of $Q$ orthogonal to the dimension vector of $V$. Guided by quiver invariant theoretic considerations, we introduce the Brascamp-Lieb operator $T_{\V,\sigma}$ associated to $(\V,\sigma)$ and study its capacity, denoted by $\capa_Q(\V, \sigma)$. When $Q$ is the $m$-subspace quiver, the capacity of quiver data is intimately related to the Brascamp-Lieb constants that occur in the $m$-multilinear Brascamp-Lieb inequality in analysis. 

We show that the positivity of $\capa_Q(\V, \sigma)$ is equivalent to the $\sigma$-semi-stability of $\V$. We also find a character formula for $\capa_Q(\V, \sigma)$ whenever it is positive. Our main tool is a quiver version of a celebrated result of Kempf-Ness on closed orbits in invariant theory. This result leads us to consider certain real algebraic varieties that carry information relevant to our main objects of study. It allows us to express the capacity of quiver data in terms of the character induced by $\sigma$ and sample points of the varieties involved. Furthermore, we use this character formula to prove a factorization of the capacity of quiver data. We also show that the existence of gaussian extremals for $(\V, \sigma)$ is equivalent to $\V$ being $\sigma$-polystable, and that the uniqueness of gaussian extremals implies that $\V$ is $\sigma$-stable. Finally, we explain how to find the gaussian extremals of a gaussian-extremisable datum $(\V, \sigma)$ using the algebraic variety associated to $(\V,\sigma)$.
\end{abstract}

\maketitle
\setcounter{tocdepth}{1}
\tableofcontents

\section{Introduction} 
\subsection{Motivation} The motivation in this paper goes back to the celebrated Brascamp-Lieb (BL) inequality in harmonic analysis. Let $m, d, d_1, \ldots, d_m \geq 1$ be integers and $\mathbf{p}=(p_1, \ldots, p_m)$ an $m$-tuple of non-negative real numbers such that $d=\sum_{j=1}^m p_jd_j$. Let $\mathbf{\V}=(\V_j)_{j=1}^m \in \prod_{j=1}^m \RR^{d_j \times d}$ be an $m$-tuple of matrices. The BL constant associated to the datum $(\mathbf{\V},\mathbf{p})$ is the best constant for which the BL inequality
\begin{equation} \label{classical-BL-ineq}
\int_{\RR^d} \prod_{j=1}^m  \left( f_j \circ \V_j \right)^{p_j}  \leq \bl(\mathbf{\V}, \mathbf{p}) \prod_{j=1}^m \left(  \int_{\RR^{d_j}} f_j\right)^{p_j}
\end{equation}
holds for all non-negative integrable functions $f_j:\RR^{d_j}\to \RR$, $j \in [m]$. The constant $\bl(\mathbf{\V}, \mathbf{p})$ can be infinite, in which case the BL inequality is rather vacuous. However, in the finite case, the Brascamp-Lieb inequality generalizes many classical inequalities in Harmonic Analysis such as the H{\" o}lder, Young's convolution, and Loomis-Whitney inequalities, just to name a few. Furthermore, the BL constants/inequalities permeate various other areas of mathematics, including convex geometry, functional analysis, and computer science. See for example \cite{Ball-1989, Bar-1998, Bar-et-al-2011, BenCarChrTao-2008, Ben-Bez-Flo-Lee-2018, BenCarTao-2006, CheDafPao-2015, GarGurOliWig-2017, Dvir-Hi-2016, DvirGarOliSol-2018}. 

A systematic study of BL constants has been undertaken by J. Bennett, A. Carbery, M. Christ, and T. Tao in \cite{BenCarChrTao-2008} where the authors prove many important results. For example, they give necessary and sufficient conditions for the finiteness of $\bl(\mathbf{\V},\mathbf{p})$. The linear maps $\V_j$, $j \in [m]$, that appear in $(\ref{classical-BL-ineq})$ can be encoded as a representation of the bipartite directed graph $\mathcal{Q}_m$, also known as the \emph{$m$-subspace quiver}, as follows:
$$\mathcal{Q}_m:~
\vcenter{\hbox{  
\begin{tikzpicture}[point/.style={shape=circle, fill=black, scale=.3pt,outer sep=3pt},>=latex]
   \node[point,label={left:$v_1$}] (1) at (-2,0) {};
   \node[point,label={right:$w_1$}] (2) at (0,1.5) {};
   \node[point,label={right:$w_2$}] (3) at (0,1) {};
   \node[point,label={right:$w_{m-1}$}] (4) at (0,-1) {};
   \node[point,label={right:$w_m$}] (5) at (0,-1.5) {};
  
   \draw[dotted] (0,.1)--(0,-.1);
  
   \path[->]
   (1) edge [bend left=15] node[midway, above] {$a_1$} (2)
   (1) edge [bend left=20] node[midway, below] {$a_2$} (3)
   (1) edge [bend right=20]  node[midway, above] {$a_{m-1}$} (4)
   (1) edge [bend right=15] node[midway, below] {$a_m$} (5);
\end{tikzpicture} 
}}
\hspace{30pt}
\mathbf{\V}:~
\vcenter{\hbox{
\begin{tikzpicture}[point/.style={shape=circle, fill=black, scale=.3pt,outer sep=3pt},>=latex]
   \node[point,label={left:$\RR^d$}] (1) at (-2,0) {};
   \node[point,label={right:$\RR^{d_1}$}] (2) at (0,1.5) {};
   \node[point,label={right:$\RR^{d_2}$}] (3) at (0,1) {};
   \node[point,label={right:$\RR^{d_{m-1}}$}] (4) at (0,-1) {};
   \node[point,label={right:$\RR^{d_m}$}] (5) at (0,-1.5) {};
  
   \draw[dotted] (0,.1)--(0,-.1);
  
   \path[->]
   (1) edge [bend left=15] node[midway, above] {$\V_1$} (2)
   (1) edge [bend left=20] node[midway, below] {$\V_2$} (3)
   (1) edge [bend right=20]  node[midway, above] {$\V_{m-1}$} (4)
   (1) edge [bend right=15] node[midway, below] {$\V_m$} (5);
\end{tikzpicture} 
}}
$$

\noindent
Then \cite[Theorem 1.13]{BenCarChrTao-2008} simply says that $\bl(\mathbf{\V},\mathbf{p})< \infty$ if and only if $\mathbf{\V}$ is a semi-stable representation of $\mathcal{Q}_m$ with respect to the weight defined by $\mathbf{p}$. Furthermore, the following comment appears in \cite[Section 4]{BenCarChrTao-2008}: ``\emph{It is likely that the deeper theory of such [quiver] representations is of relevance to this [Brascamp-Lieb] theory, but we do not pursue these connections here}.'' 

In this paper, we study BL constants within the general framework of quiver invariant theory.

\subsection{Our results} We briefly recall just enough terminology to state our main results, with more detailed background found in Section \ref{BL-operators-sec}. Let $Q$ be a connected quiver with set of vertices $Q_0$ and set of arrows $Q_1$. For an arrow $a \in Q_1$, we denote by $ta$ and $ha$, its tail and head, respectively. We represent $Q$ as a directed graph with set of vertices $Q_0$ and directed edges $a:ta \to ha$ for every $a \in Q_1$. A real representation $\V$ of $Q$ assigns a finite-dimensional real vector space $V(x)$ to every vertex $x \in Q_0$ and a linear map $\V(a): \V(ta) \to \V(ha)$ to every arrow $a \in Q_1$. After fixing bases for the vector spaces $\V(x)$, $x \in Q_0$, we often think of the linear maps $\V(a)$, $a \in Q_1$, as matrices of appropriate size. The dimension vector of a representation $\V$ of $Q$ is $\ddim \V:=(\dim_{\RR} \V(x))_{x \in Q_0} \in \NN^{Q_0}$.

Let $\sigma \in \ZZ^{Q_0}$ be an integral weight of $Q$. A representation $\V$ of $Q$ is said to be \emph{$\sigma$-semi-stable} if $\sigma \cdot \ddim \V=0$ and $\sigma \cdot \ddim \V'\leq 0$ for all subrepresentations $\V' \leq \V$. We say that $\V$ is \emph{$\sigma$-stable} if $\sigma \cdot \ddim \V=0$ and $\sigma \cdot \ddim \V'< 0$ for all proper subrepresentations $\V'$ of $\V$. We call a representation \emph{$\sigma$-polystable} if it is a finite direct sum of $\sigma$-stable representations. 

For our purposes, we can simply assume that $Q$ is bipartite (see Remark \ref{general-bipartite-case-rmk}). This means that $Q_0$ is the disjoint union of two subsets $Q^+_0=\{v_1, \ldots, v_n\}$ and $Q^-_0=\{w_1, \ldots, w_m\}$, and all arrows in $Q$ go from $Q^+_0$ to $Q^-_0$. Furthermore, we assume that $\sigma$ is positive on $Q_0^+$, and negative on $Q_0^-$. 

Let $\dd \in \NN^{Q_0}$ be a dimension vector such that $\sigma \cdot \dd=0$, and let $\V$ be a $\dd$-dimensional representation of $Q$ with $\V(x)=\RR^{\dd(x)}, \forall x\in Q_0$, and $\V(a) \in \RR^{\dd(ha) \times \dd(ta)}, \forall a \in Q_1$. Guided by invariant theoretic considerations and \cite[Construction 4.2]{GarGurOliWig-2017}, we associate to the quiver datum $(\V, \sigma)$, the so-called \emph{BL operator} $T_{\V,\sigma}$ (see Definition \ref{BL-operator-def}). This is a completely positive operator whose Kraus operators are certain blow-ups of the matrices $\V(a), a \in Q_1$. 

We define the \emph{capacity} of $(\V,\sigma)$, denoted by $\capa_Q(\V,\sigma)$, to be the capacity of the operator $T_{\V,\sigma}$. The determinantal formula for $\capa_Q(\V,\tup)$ in Lemma \ref{cap-compute-lemma} leads us to the definition of the \emph{BL constant} $\bl_Q(\V,\tup)$ associated to $(\V, \tup)$ (see Definition \ref{BL-defn-quivers}). In fact, when $Q=\mathcal{Q}_m$ is the $m$-subspace quiver, we recover the classical BL constants.

Our first result gives necessary and sufficient conditions for the positivity of the capacity of a quiver datum. 

\begin{theorem}\label{cap-semi-stab-thm} 
Let $Q$ be a bipartite quiver and $(\V,\sigma)$ a quiver datum. Then
$$
\capa_Q(\V,\sigma)>0 \Longleftrightarrow \V \text{~is~} \sigma\text{~-semi-stable}.
$$ 
\end{theorem}
\noindent

In \cite[Corollary 3.17]{GarGurOliWig-2017}, the authors have found a deterministic polynomial time algorithm for deciding the positivity of the capacity of a completely positive operator. This algorithm combined with Theorem \ref{cap-semi-stab-thm} yields a $poly(b,N)$ time (deterministic) algorithm to check if $\V$ is $\sigma$-semi-stable where $b$ is the total bit size of $\V$ and $N=\sum_{i=1}^n \sigma(v_i) \dd(v_i)$. The importance of the existence of such an algorithm stems from the fact that, in general, quiver semi-stability requires one check a number of linear homogeneous inequalities that can grow exponentially. 

In Lemma \ref{cap-compute-lemma}, we show that $\capa_Q(\V,\sigma)$ is the infimum of certain determinantal expressions where the infimum is taken over all positive definite matrices $Y_j \in \RR^{\dd(w_j) \times \dd(w_j)}$, $j \in [m]$. We say that a quiver datum $(\V,\sigma)$ is \emph{gaussian-extremisable} if the infimum defining $\capa_Q(\V,\sigma)$ is attained for some positive definite matrices $Y_j \in \RR^{\dd(w_j)\times \dd(w_j)}$, $j \in [m]$. If this is the case, we call such an $m$-tuple $(Y_1, \ldots, Y_m)$ a \emph{gaussian extremiser} for $(\V,\sigma)$.

One of our main goals in this paper is to find a constructive method for computing $\capa_Q(\V,\sigma)$, and gaussian extremisers whenever $(\V,\sigma)$ is gaussian-extremisable. To this end, we introduce the notion of a \emph{geometric quiver datum}: We say that $(\V,\sigma)$ is \emph{geometric} if the corresponding operator $T_{\V,\sigma}$ is doubly-stochastic (see Definition \ref{quiver-geom-def}). Consequently, the capacity of such a datum is always one (see \cite[Proposition 2.8 and Corollary 3.4]{GarGurOliWig-2015}). Geometric quiver data are also intimately related to the Kempf-Ness theorem on closed orbits in invariant theory.

Our next result gives a quiver invariant process that transforms an arbitrary quiver datum $(\V,\sigma)$ with $\capa_Q(\V,\sigma)>0$ into a geometric one. In particular it leads to a character formula for $\capa_Q(\V,\sigma)$. To state this result, we need to introduce a few more concepts. The \emph{representation space} of $\dd$-dimensional representations of $Q$ is the affine space $\rep(Q,\dd)=\prod_{a \in Q_1} \RR^{\dd(ha)\times \dd(ta)}$. It is acted upon by the change of base group $\GL(\dd)=\prod_{x \in Q_0}\GL(\dd(x), \RR)$ by simultaneous conjugation. The \emph{character of $\GL(\dd)$ induced by $\sigma$} is $\chi_{\sigma}:\GL(\dd) \to \RR^{\times}=\RR \setminus \{0\}$, $\chi_{\sigma}(A)=\prod_{x \in Q_0}\det(A(x))^{\sigma(x)}$ for all $A=(A(x))_{x \in Q_0} \in \GL(\dd)$. We denote by $\GL(\dd)_{\sigma}$ the kernel of $\chi_{\sigma}$.

\begin{theorem} \label{main-thm-2} Let $Q$ be a bipartite quiver, $\dd \in \NN^{Q_0}$ a dimension vector of $Q$, and $\sigma \in \ZZ^{Q_0}$ an integral weight of $Q$ orthogonal to $\dd$. Assume that $\sigma$ is positive on $Q_0^{+}$ and negative on $Q_0^{-}$. 

\begin{enumerate}
\item (\textbf{Kempf-Ness theorem for real quiver representations}) For a $\sigma$-semi-stable representation $\V \in \rep(Q,\dd)$, consider the real algebraic variety
$$
\G_{\sigma}(\V):=\{A \in \GL(\dd) \mid (A\cdot \V, \sigma)\text{~is a geometric datum}\}.
$$ 
Then
$$
\G_{\sigma}(\V) \neq \emptyset \Longleftrightarrow \V \text{~is~}\sigma-\text{polystable}.
$$

\bigskip

\item (\textbf{A character formula for capacity}) Let $\V \in \rep(Q,\dd)$ be a $\sigma$-semi-stable representation. Then there exists a $\sigma$-polystable representation $\widetilde{\V}$ such that $\widetilde{\V} \in \overline{\GL(\dd)_{\sigma}\V}$. Furthermore, for any such $\widetilde{\V}$, the following formula holds
$$
\capa_Q(\V,\sigma)=\capa_Q(\widetilde{\V},\sigma)=\chi_{\sigma}(A)^2, \forall A \in \G_{\sigma}(\widetilde{\V}).
$$ 	

\bigskip

\item (\textbf{Factorization of capacity}) Let $\V \in \rep(Q,\dd)$ be a representation such that 
$$
\V(a)=\left(
\begin{matrix}
\V_1(a) & X(a)\\
0&\V_2(a)
\end{matrix}\right),
\forall a \in Q_1,
$$
where $\V_i \in \rep(Q,\dd_i)$, $i \in \{1,2\}$, are representations of $Q$, and $X(a) \in \RR^{\dd_1(ha)\times \dd_2(ta)}, \forall a \in Q_1$. If $\sigma \cdot  \ddim \V_1=0$ then
$$
\capa_Q(\V,\sigma)=\capa_Q(\V_1,\sigma)\cdot \capa_Q(\V_2,\sigma).
$$

\bigskip		
		
\item (\textbf{Gaussian extremisers: existence}) A quiver datum $(\V, \sigma)$ with $\capa_Q(\V, \sigma)>0$ is gaussian-extremisable if and only if $\V$ is $\sigma$-polystable. If this is the case then the gaussian extremisers of $(\V, \sigma)$ are the $m$-tuples of matrices 
$$
(A(w_j)^T \cdot A(w_j))_{j \in [m]} \text{~with~} A \in \G_{\sigma}(\V).$$ 

\bigskip

\item (\textbf{Gaussian extremisers: uniqueness}) If a quiver datum $(\V, \sigma)$ with  $\capa_Q(\V, \sigma)>0$ has unique gaussian extremisers (up to scaling) then $\V$ is $\sigma$-stable. Conversely if $\V$ is $\sigma$-stable with a one-dimensional space of endomorphisms then $(\V, \sigma)$ has unique gaussian extremisers (up to scaling).
\end{enumerate}
\end{theorem}

In Theorem \ref{sumarry-BL-consts-thm}, we reformulate the results above in terms of BL constants for bipartite quivers. When $Q$ is the $m$-subspace quiver $\mathcal{Q}_m$, we recover the main results of Bennet-Carbery-Christ-Tao \cite{BenCarChrTao-2008} on the classical BL constants with rational $m$-exponents.

On the computational side, the character formula above opens up the possibility of computing capacities (BL-constants) and gaussian extremisers for quiver data via algebraic sampling algorithms (see for example \cite{BasPolRoy-2006}).

\subsection*{Acknowledgment} The first author would like to thank Harm Derksen and University of Michigan for their kind hospitality during his one week visit in July 2018. C. Chindris is supported by Simons Foundation grant $\# 711639$. H. Derksen is supported by NSF grant DMS-1601229. The authors would also like to thank Edward Duran, Dan Edidin, Cole Franks, Visu Makam, Peter Pivovarov, and Petros Valettas for many helpful conversations on the subject of the paper. 

\section{Brascamp-Lieb operators and the capacity of quiver representations} \label{BL-operators-sec}
Throughout, we work over the field $\RR$ of real numbers and denote by $\NN=\{0,1,\dots \}$. For a positive integer $L$, we denote by $[L]=\{1, \ldots, L\}$.

A quiver $Q=(Q_0,Q_1,t,h)$ consists of two finite sets $Q_0$ (\emph{vertices}) and $Q_1$ (\emph{arrows}) together with two maps $t:Q_1 \to Q_0$ (\emph{tail}) and $h:Q_1 \to Q_0$ (\emph{head}). We represent $Q$ as a directed graph with set of vertices $Q_0$ and directed edges $a:ta \to ha$ for every $a \in Q_1$. Throughout we assume that our quivers are connected, meaning that the underlying graph of $Q$ is connected.

A representation of $Q$ is a family $\V=(\V(x), \V(a))_{x \in Q_0, a\in Q_1}$ where $\V(x)$ is a finite-dimensional $\RR$-vector space for every $x \in Q_0$, and $\V(a): \V(ta) \to \V(ha)$ is a $\RR$-linear map for every $a \in Q_1$. A \emph{subrepresentation} $\V'$ of $\V$, written as $V' \leq \V$, is a representation of $Q$ such that $\V'(x) \leq_{\RR} \V(x)$ for every $x \in Q_0$, and $\V(a)(\V'(ta)) \subseteq \V'(ha)$ and $\V'(a)$ is the restriction of $\V(a)$ to $\V(ta)$ for every arrow $a \in Q_1$.

The dimension vector $\ddim \V \in \NN^{Q_0}$ of a representation $\V$  is defined by $\ddim \V(x)=\dim_{\RR} \V(x)$ for all $x \in Q_0$. By a dimension vector of $Q$, we simply mean a $\ZZ_{\geq 0}$-valued function on the set of vertices $Q_0$. For two vectors $\theta, \beta \in \RR^{Q_0}$, we define $\theta \cdot \beta=\sum_{x \in Q_0} \theta(x)\beta(x)$. 

Let $\dd \in \NN^{Q_0}$ be a dimension vector. The representation space of $\dd$-dimensional representations of $Q$ is the affine space
$$
\rep(Q,\dd)=\prod_{a \in Q_1}\RR^{\dd(ha)\times \dd(ta)}.
$$
The change-of-base group $\GL(\dd)=\prod_{x \in Q_0}\GL(\dd(x), \RR)$ acts on $\rep(Q,\dd)$ by simultaneous conjugation, i.e. for $A=(A(x))_{x \in Q_0}$ and $\V=(\V(a))_{a \in Q_1}$, we have that
$$(A\cdot \V)(a)=A(ha)\cdot \V(a) \cdot A(ta)^{-1}, \forall a \in Q_1.
$$
Note that there is a bijective correspondence between the isomorphism classes of representations of $Q$ of dimension vector $\dd$ and the $\GL(\dd)$-orbits in $\rep(Q,\dd)$.

From now on, we assume that $Q$ is bipartite. This means that $Q_0$ is the disjoint union of two subsets $Q^+_0$ and $Q^-_0$, and all the arrows in $Q$ go from $Q^+_0$ to $Q^-_0$. Write $Q_0^{+}=\{v_1, \ldots, v_n\}$ and $Q_0^{-}=\{w_1,\ldots, w_m\}$. 

Let us fix an integral weight $\sigma \in \ZZ^{Q_0}$ such that $\sigma$ is positive on $Q^+_0$ and negative on $Q^-_0$. Define
$$\sigma_{+}(v_i)=\sigma(v_i), \forall i \in [n], \text{~and~}
\sigma_{-}(w_j)=-\sigma(w_j), \forall j \in [m].
$$
Let $\HH(\dd)=\{\sigma \in \RR^{Q_0} \mid \sum_{x \in Q_0}\sigma(x)\dd(x)=0\}$
be the space of real weights of $Q$ orthogonal to $\dd$, and let us assume that $\sigma \in \HH(\dd)$. This is equivalent to
$$
N:=\sum_{i=1}^n \sigma_{+}(v_i)\dd(v_i)=\sum_{j=1}^m \sigma_{-}(w_j)\dd(w_j).
$$

For $i\in [n]$ and $j \in [m]$, we denote the set of all arrows in $Q$ from $v_i$ to $w_j$  by $\ar_{i,j}$. If there are no arrows from $v_i$ to $w_j$, we define $\ar_{i,j}$ to be the set consisting of the symbol $\0_{ij}$.

Let $M:=\sum_{j=1}^m \sigma_{-}(w_j) \text{~and~} M':=\sum_{i=1}^n \sigma_{+}(v_i).$
For each $j \in [m]$ and $i \in [n]$, define 

$$
\jj:=\{q \in \ZZ \mid \sum_{k=1}^{j-1} \sigma_{-}(w_k) < q \leq  \sum_{k=1}^{j} \sigma_{-}(w_k) \},
$$
and 
$$
\ii:=\{r \in \ZZ \mid \sum_{k=1}^{i-1} \sigma_{+}(v_k) < r \leq  \sum_{k=1}^{i} \sigma_{+}(v_k) \}.
$$
In what follows, we consider $M \times M'$ block matrices of size $N \times N$ such that for any two indices $q \in \jj$ and $r \in \ii$, the $(q,r)$-block-entry is a matrix of size $\dd(w_j)\times \dd(v_i)$. Set
$$
\s:= \{(i,j,a,q,r) \mid i \in [n], j \in [m],\\ a \in \ar_{i,j}, \\ q \in \jj, r \in \ii \}.
$$

\smallskip
Now, let $\V\in \rep(Q,\dd)$ be a $\dd$-dimensional representation of $Q$. For each $(i,j,a,q,r) \in \s$, let $\V^{i,j,a}_{q,r}$ be the $M \times M'$ block matrix whose $(q,r)$-block-entry is $\V(a) \in \RR^{\dd(w_j)\times \dd(v_i)}$, and all other entries are zero. The convention is that if $a=\0_{ij} \in \ar_{i,j}$ then $\V(a)$ is the zero matrix of size $\dd(w_j)\times \dd(v_i)$; hence, if there are no arrows from $v_i$ to $w_j$ then $\V^{i,j,a}_{q,r}$ is the zero matrix of size $N \times N$. 

\begin{rmk}
The $N \times N$ matrices $\V^{i,j,a}_{q,r}$, where $(i,j,a,q,r) \in \s$ and $\V \in \rep(Q,\dd)$, play a key role in the theory of semi-invariants of acyclic quivers. Specifically, let $t^{i,j,a}_{q,r}$, $(i,j,a,q,r) \in \s$, be indeterminate variables. Then, assuming that $K=\CC$, the coefficients of the polynomial
$$
\det\left( \sum_{(i,j,a,q,r)} t^{i,j,a}_{q,r} \V^{i,j,a}_{q,r} \right) \in K[\rep(Q,\dd)][t^{i,j,a}_{q,r}: (i,j,a,q,r) \in \s]
$$
span the weight space of semi-invariants $\SI(Q,\dd)_{\sigma}$. For more details, see \cite[Section 5]{HarmVisu-2017} and the reference therein.
\end{rmk}

Inspired by \cite[Construction 4.2]{GarGurOliWig-2017}, we now introduce Brascamp-Lieb operators for arbitrary quivers.

\begin{definition} \label{BL-operator-def} Let $\V \in \rep(Q,\dd)$ be a $\dd$-dimensional representation of $Q$.
\begin{enumerate}
\item The \emph{Brascamp-Lieb} operator $T_{\V, \sigma}$ associated to $(\V, \sigma)$ is defined to be the completely positive operator with Kraus operators $\V^{i,j,a}_{q,r}$, $(i,j,a,q,r) \in \s$, i.e.
\begin{align*}
T_{\V,\sigma}: \RR^{N \times N} & \to \RR^{N \times N}\\
X& \to T_{\V,\sigma}(X):=\sum_{(i,j,a,q,r)}(\V^{i,j,a}_{q,r})^T\cdot X \cdot \V^{i,j,a}_{q,r}
\end{align*}

\item The \emph{capacity} $\capa_Q(\V,\sigma)$ of $(\V,\sigma)$ is defined to be the capacity of $T_{V, \sigma}$, i.e.
$$
\capa_Q(\V,\sigma):=\inf \{\Det(T_{\V,\sigma}(X)) \mid  X \in \s^{+}_N,  \Det(X)=1 \}.
$$
(Here, for a given positive integer $d$, we denote by $\s^{+}_d$ the set of all $d \times d$ (symmetric) positive definite real matrices.)
\end{enumerate}
\end{definition}

\begin{rmk} 
\begin{enumerate}
\item We point out that completely positive operators are usually defined over $\CC$, and the infimum defining their capacity is taken over positive definite complex matrices. However, if $T$ is defined by real Kraus operators then one can simply work with positive definite \emph{real} matrices in the definition of the capacity of $T$ (see \cite[Remark 2.7]{GarGurOliWig-2017}).  

\item Any completely positive operator $T$ with Kraus operators $A_1, \ldots, A_l$ can be viewed as a Brascamp-Lieb operator for the generalized Kronecker quiver with $l$ arrows, representation $\V=(A_1, \ldots, A_l)$, and weight $\sigma=(1,-1)$. However, it is important to keep $Q$ arbitrary and not simply reduce the considerations to generalized Kronecker quivers. Indeed, Theorem \ref{cap-semi-stab-thm} allows us to interpret quiver semi-stability for arbitrary bipartite (and not just generalized Kronecker) quivers in terms of the positivity of the capacity of BL operators. As already mentioned before, this in turn leads to a deterministic polynomial time algorithm for checking whether a representation $\V$ of a bipartite quiver is semi-stable with respect to an integral weight $\sigma$. 

\item Completely positive operators whose Kraus operators look similar to our $\V^{i,j,a}_{q,r}$ are also considered in \cite[Section 3]{Franks-2018}. However, our definition of $T_{\V,\sigma}$ is based on quiver invariant theoretic considerations, and the overall approach in this paper is different than that in \emph{loc. cit.}.
\end{enumerate}
\end{rmk}

\begin{remark}\label{general-bipartite-case-rmk} Brascamp-Lieb operators can be defined for quivers which are not necessarily bipartite. Specifically, let $Q=(Q_0,Q_1,t,h)$ be an arbitrary acyclic quiver and $\dd \in \NN^{Q_0}$ a dimension vector. Let $Q^+_0=\{v_1, \ldots, v_n \}$ and $Q^-_0=\{w_1, \ldots, w_m\}$ be two disjoint subsets of $Q_0$, and let $\sigma \in \ZZ^{Q_0} \cap \HH(\dd)$ be an integral weight such that $\sigma$ is positive on $Q^+_0$, negative on $Q^-_0$, and zero elsewhere.

Let $Q^{\pm}$ be the bipartite quiver with set of vertices $Q^{+}_0 \cup Q^{-}_0$. For every oriented path $p$ in $Q$ from $v_i$ to $w_j$, we define an arrow $a_p$ in $Q^{\pm}$ from $v_i$ to $w_j$. Given a representation $\V$ of $Q$, let $\V^{\pm}$ be the representation of $Q^{\pm}$ defined by
\begin{itemize}
\item $\V^{\pm}(v_i)=\V(v_i)$, $\V^{\pm}(w_j)=\V(w_j)$ for all $i \in [n]$, $j \in [m]$, and
\smallskip
\item $\V^{\pm}(a_p)=\V(p)$ for every arrow $a_p$ in $Q^{\pm}$.
\end{itemize}	
We then simply define $T_{V,\sigma}:=T_{V^{\pm},\sigma}$, and $\capa_Q(\V,\sigma):=\capa_{Q^{\pm}}(\V^{\pm}, \sigma)$.
\end{remark}

To prove our first Theorem \ref{cap-semi-stab-thm}, we require the following very useful general criterion addressing the positivity of the capacity of a completely positive operator. 

\begin{lemma} (\cite[Corollary 3.15]{GarGurOliWig-2017}) \label{cap-pos-general-lemma} Let $T:\RR^{N \times N} \to \RR^{N \times N}$ be a completely positive operator. Then $\capa(T)>0$ if and only if 
$$
\rk(X) \leq \rk T^*(X), \forall  X \succeq 0.
$$
\end{lemma}

We point out that the proof below is an adaptation of that of \cite[Lemma 4.4]{GarGurOliWig-2017} to our general quiver set-up. Nonetheless, we include it for completeness and convenience of the reader. 

\begin{proof}[Proof of Theorem \ref{cap-semi-stab-thm}] We will prove that $\capa_Q(\V,\sigma)>0$ if and only if
\begin{equation} \label{semi-stab-ineq}
\sum_{i=1}^n \sigma_{+}(v_i) \dim \V'(v_i) \leq \sum_{j=1}^m \sigma_{-}(w_j) \dim \left( \sum_{i=1}^n \sum_{a \in \ar_{i,j}}\V(a)(\V'(v_i)) \right),
\end{equation}
for all subspaces $\V'(v_i) \leq \RR^{\dd(v_i)}$, $\forall i \in [n]$. The latter is easily seen to be equivalent to $\V$ being $\sigma$-semi-stable.

We know from Lemma \ref{cap-pos-general-lemma} that 
$$
\capa(\V, \sigma)>0 \Longleftrightarrow \rk(X) \leq \rk(T^*_{\V, \sigma}(X)), \forall  N \times N \text{~matrices~} X \succeq 0.
$$
\noindent
By definition, 
$$
T^*_{\V, \sigma}(X)=\sum_{(i,j,a,q,r)}\V^{i,j,a}_{q,r}\cdot X \cdot (\V^{i,j,a}_{q,r})^T, \forall X \in \RR^{N \times N}.
$$ 
Viewing each $N \times N$ matrix $X$ as an $M'\times M'$ block matrix, we get that for each $(i,j,a,q,r) \in \s$, the matrix 
$$
\V^{i,j,a}_{q,r} \cdot X \cdot (\V^{i,j,a}_{q,r})^T
$$ 
has an $M \times M$ block matrix structure whose $(q,q)$-block entry is
$$
\V(a)\cdot X_{rr}\cdot (\V(a))^T,
$$
and all other blocks are zero. So, $T^*_{\V, \sigma}(X)$ is the $M \times M$ block-diagonal matrix whose $(q,q)$-block-diagonal entry is
$$
\sum_{i=1}^n \sum_{a \in \ar_{i,j}} \V(a) (\sum_{r \in \ii} X_{rr}) \V(a)^T,
$$
for all $q \in \jj$ and $j \in [m]$. It now follows that
\begin{align*}
 \rk(X) \leq & \rk(T^*_{\V, \sigma}(X)), \forall N \times N \text{~matrices~} X \succeq 0\\  
&\rotatebox[origin=c]{-90}{$\Leftrightarrow$} \notag \\ 
\sum_{i=1}^n \sum_{r \in \ii} \rk(X_r) \leq \sum_{j=1}^m\sigma_{-}(w_j) & \rk \left( \sum_{i=1}^n \sum_{a \in \ar_{i,j}} \V(a) \left( \sum_{r \in \ii} X_r \right) \V(a)^T \right) \hspace{25pt} (\star)
\end{align*}
for all positive semi-definite matrices $X_r \in \RR^{\dd(v_i) \times \dd(v_i)}$ with $r \in \ii$ and $i \in [n]$. 

\bigskip
\noindent
($\Longrightarrow$) Let us assume that the linear homogeneous inequalities $(\star)$ hold for all positive semi-definite matrices $X_r \in \RR^{\dd(v_i) \times \dd(v_i)}$ with $r \in \ii$ and $i \in \{1, \ldots, n \}$.  

Let $\V'(v_i) \leq \RR^{\dd(v_i)}$, $i \in [n] $, be arbitrary subspaces. Choose an orthonormal basis $\{u^i_1, \ldots, u^i_{\dd'(i)}\}$ for each $\V'(v_i)$ and set
$$
X_r=\sum_{l=1}^{\dd'(v_i)} u_l^i \cdot (u_l^i)^T,
$$
for every $r \in \ii$. Plugging these matrices into $(\star)$, we get
\begin{equation} \label{eq1}
\sum_{i=1}^n \sigma_{+}(v_i)\cdot \dim \V'(v_i) \leq \sum_{j=1}^m \sigma_{-}(w_j) \rk \left( \sum_{i=1}^n \sum_{a \in \ar_{i,j}} \sum_{r,l} \V(a)u^i_l (\V(a)u^i_l)^T \right).
\end{equation}

But each $\rk \left( \sum_{i=1}^n \sum_{a \in \ar_{i,j}} \sum_{r,l} \V(a)u^i_l (\V(a)u^i_l)^T \right)$ equals the dimension of the space spanned by the vectors $\V(a)u^i_l$, i.e.
\begin{equation} \label{eq2}
\rk \left( \sum_{i=1}^n \sum_{a \in \ar_{i,j}} \sum_{r,l} \V(a)u^i_l (\V(a)u^i_l)^T \right)=\dim \left(\sum_{i=1}^n \sum_{a \in \ar_{i,j}} \V(a)(\V'(v_i))  \right).
\end{equation}

It now follows from (\ref{eq1}) and (\ref{eq2}) that
$$
\sum_{i=1}^n \sigma_{+}(v_i) \dim \V'(v_i) \leq \sum_{j=1}^m \sigma_{-}(w_j) \dim \left( \sum_{i=1}^n \sum_{a \in \ar_{ij}}\V(a)(\V'(v_i)) \right).
$$

\bigskip
\noindent 
$(\Longleftarrow)$ Let $X_r \in \RR^{\dd(v_i)\times \dd(v_i)}$, $r \in \ii$, $i \in \{1, \ldots, n \}$, be arbitrary positive semi-definite matrices. For each such $r$ and $i$, let $\{u^{i,r}_1, \ldots, u^{i,r}_{\dd_{i,r}} \}$ be an orthonormal set of vectors in $\RR^{\dd(v_i)}$ such that
$$
X_r=\sum_{l=1}^{\dd_{i,r}} \lambda^{i,r}_l u^{i,r}_l \cdot (u^{i,r}_l)^T,
$$
with the $\lambda^{i,r}_l>0$; in particular, $\rk(X_r)=\dd_{i,r}$. Now, define
$$
\V'(v_i)=\Span \left( \sqrt{\lambda^{i,r}_l} \cdot u^{i,r}_l \mid r \in \ii, 1 \leq l \leq \dd_{i,r} \right) \leq \RR^{\dd(v_i)}.
$$
Working with these subspaces in $(\ref{semi-stab-ineq})$, we get that $(\star)$ holds all positive semi-definite matrices $X_r$. In other words, $\capa(T_{\V, \sigma})>0$.
\end{proof}

\section{Brascamp-Lieb constants from capacity of quiver representations} 
Let $Q=(Q_0,Q_1,t,h)$ be a bipartite quiver with set of source vertices $Q_0^{+}=\{v_1, \ldots, v_n\}$ and set of sink vertices $Q_0^{-}=\{w_1,\ldots, w_m\}$. Let $\dd \in \NN^{Q_0}$ be a dimension vector of $Q$ and $\sigma \in \HH(\dd) \cap \ZZ^{Q_0}$ a weight such that $\sigma$ is positive on $Q_0^+$ and negative on $Q_0^-$. Recall the notation from Section \ref{BL-operators-sec}:

\begin{itemize}
\smallskip
\item $\sigma_{+}(v_i)=\sigma(v_i), \forall i \in [n]$ and $\sigma_{-}(w_j)=-\sigma(w_j), \forall j \in [m]$;

\smallskip
\item $N=\sum_{i=1}^n \sigma_{+}(v_i)\dd(v_i)=\sum_{j=1}^m \sigma_{-}(w_j)\dd(w_j)$;

\smallskip
\item $\ar_{i,j}$ is the set of arrows from $v_i$ to $w_j$ in $Q$ for all $i \in\{1, \ldots, n \}$ and $j \in \{1, \ldots, m \}$;

\smallskip
\item
$M:=\sum_{j=1}^m \sigma_{-}(w_j)$, and $M':=\sum_{i=1}^n \sigma_{+}(v_i)$.
\end{itemize}

\noindent
For each $j \in \{1, \ldots, m \}$ and $i \in \{1, \ldots, n \}$, we furthermore define 
$$
\jj:=\{q \in \ZZ \mid \sum_{k=1}^{j-1} \sigma_{-}(w_k) < q \leq  \sum_{k=1}^{j} \sigma_{-}(w_k) \},
$$
and 
$$
\ii:=\{r \in \ZZ \mid \sum_{k=1}^{i-1} \sigma_{+}(v_k) < r \leq  \sum_{k=1}^{i} \sigma_{+}(v_k) \}.
$$

In Lemma \ref{cap-compute-lemma} below we provide a more explicit formula for the capacity of a quiver datum$(\V, \sigma)$ that will lead us to the definition of the BL constant associated to $(\V, \sigma)$. For this, we recall the following well-known facts. Let $X$ be a positive semi-definite $N\times N$ matrix, viewed as an $M \times M$ block matrix. For each $j \in [m]$ and $q \in \jj$, denote by $X_{qq}$ the $(q,q)$-block-diagonal entry of $X$; it is of size $\dd(w_j) \times \dd(w_j)$. Then we have that

\begin{equation} \label{det-diag}
\det(X) \leq \prod_{j=1}^m \prod_{q \in \jj} \det(X_{qq})
\end{equation}
Next, for any $j \in \{1, \ldots, m \}$, set
$$
Y_j:={\sum_{q \in \jj} X_{qq} \over \sigma_{-}(w_j)}
$$ 
Then a generalization of Hadamard's inequality yields
\begin{equation}\label{Hadamard}
\prod_{q \in \jj} \det(X_{qq}) \leq \det(Y_j)^{\sigma_{-}(w_j)}
\end{equation}

\noindent
We are now ready to prove the following formula for $\capa_Q(\V, \sigma)$  for $\V \in \rep(Q,\dd)$.

\begin{lemma} \label{cap-compute-lemma} Let $\V \in \rep(Q,\dd)$ be a representation of $Q$. Then
\begin{align*}
&\capa_Q(\V,\sigma)= \\
&=\inf\left \{ { \prod_{i=1}^n \det \left( \sum_{j=1}^m \sigma_{-}(w_j)\left( \sum_{a \in \ar_{i,j}} \V(a)^T \cdot Y_j \cdot \V(a)  \right) \right)^{\sigma_{+}(v_i)}   \over \prod_{j=1}^m \det(Y_j)^{\sigma_{-}(w_j)}} \;\middle|\; Y_j \in \s^{+}_{\dd(w_j)}    \right \}.
\end{align*}
Furthermore, if the weight $\sigma$ is so that $\sigma_{+}(v_1)= \ldots= \sigma_{+}(v_n)=\omega>0$ then 
$$
\capa_Q(\V,\sigma)={1 \over \omega^{-N}} \cdot \left( \inf \left \{ {\prod_{i=1}^n \det\left( \sum_{j=1}^m p_j \left(\sum_{a \in \ar_{i,j}}\V(a)^T\cdot Y_j \cdot \V(a) \right) \right) \over \prod_{j=1}^m \det(Y_j)^{p_j}} \;\middle|\; Y_j \in \s^{+}_{\dd(w_j)} \right \} \right)^{\omega},
$$
where $p_j=-{\sigma(w_j)\over \omega}$ for all $j \in [m]$.
\end{lemma}

\begin{proof}
We have that
\begin{align*}
&\capa_Q(\V, \sigma)= \inf \{ \det(T_{\V,\sigma}(X)) \mid X \in \s^+_N, \det(X)=1 \}\\
=&\inf \left \{ \prod_{i=1}^n \det \left( \sum_{j=1}^m \sum_{a \in \ar_{i,j}}  \V(a)^T \left( \sum_{q \in \jj} X_{qq} \right) \V(a)   \right)^{\sigma_{+}(v_i)} \;\middle|\; X \in \s^{+}_{N} , \det(X)=1 \right \} \\
\stackrel{(iii)}{=}&\inf \left \{ \prod_{i=1}^n \det \left( \sum_{j=1}^m \sum_{a \in \ar_{i,j}}  \V(a)^T \left( \sum_{q \in \jj} X_q \right) \V(a)   \right)^{\sigma_{+}(v_i)} \;\middle|\; X_q \in \s^{+}_{\dd(w_j)},  \prod_{j=1}^m \prod_{q \in \jj}\det(X_q)=1 \right \} \\
\stackrel{(iv)}{=}&\inf \left \{ \prod_{i=1}^n \det \left( \sum_{j=1}^m\sigma_{-}(w_j) \left( \sum_{a \in \ar_{i,j}}  \V(a)^T \cdot Y_j \cdot \V(a) \right)  \right)^{\sigma_{+}(v_i)} \;\middle|\; Y_j \in \s^{+}_{\dd(w_j)},  \prod_{j=1}^m \det(Y_j)^{\sigma_{-}(w_j)}=1 \right \} \\
\stackrel{(v)}{=}&\inf \left \{ { \prod_{i=1}^n \det \left( \sum_{j=1}^m \sigma_{-}(w_j) \left( \sum_{a \in \ar_{i,j}}  \V(a)^T \cdot Y_j \cdot \V(a) \right)   \right)^{\sigma_{+}(v_i)} \over \prod_{j=1}^m \det(Y_j)^{\sigma_{-}(w_j)} } \;\middle|\; Y_j \in \s^{+}_{\dd(w_j)} \right \}
\end{align*}
To prove the third equality above, one can simply use $(\ref{det-diag})$. Indeed, it is clear that the infimum displayed on the second line above is less than or equal to that on the third line. To prove the reverse inequality, let $X$ be a positive definite $N \times N$ real matrix with $\det(X)=1$ and let us denote by $X_q$ the block-diagonal entries of $X$. Then, by $(\ref{det-diag})$, we have that 
$$
1=\det(X) \leq C:=\prod_{j=1}^m \prod_{q \in \jj}\det(X_q).
$$ 
Setting $\widetilde{X}_q={1 \over \sqrt[N]{C}}X_q$, we get that $\prod_{j=1}^m \prod_{q \in \jj}\det(\widetilde{X}_q)=1$, and
\begin{align*}
\prod_{i=1}^n \det &  \left( \sum_{j=1}^m \sum_{a \in \ar_{i,j}}  \V(a)^T \left( \sum_{q \in \jj} \widetilde{X}_q \right)  \V(a) \right)^{\sigma_{+}(v_i)}=\\
&={1 \over C} \prod_{i=1}^n \det \left( \sum_{j=1}^m \sum_{a \in \ar_{i,j}}  \V(a)^T \left( \sum_{q \in \jj} X_q \right) \V(a)   \right)^{\sigma_{+}(v_i)} 
\end{align*}
This now gives get the reverse inequality, proving the third equality above. For $(iv)$, one can simply use the generalized Hadamard's inequality $(\ref{Hadamard})$. For $(v)$, simply work with ${Y_j \over \sqrt[N]{\prod_{j=1}^m \det(Y_j)^{\sigma_{-}(w_j)}}}$, $j \in \{1, \ldots, m \}$, in the line above, where $Y_j \in \Mat_{\dd(w_j)\times \dd(w_j)}$, $j \in [m]$, are arbitrary positive definite matrices.

The formula for $\capa_Q(\V, \sigma)$ when $\sigma_{+}$ is constant follows  immediately from the computations above.
\end{proof}

Let $\chi_{\sigma}:\GL(\dd)\to \RR^{\times}$ be the character induced by $\sigma$, i.e. $\chi_{\sigma}(A)=\prod_{x \in Q_0}\det(A(x))^{\sigma(x)}$ for all $A=(A(x))_{x \in Q_0} \in \GL(\dd)$, and denote its kernel by $\GL(\dd)_{\sigma}$. As a consequence of the lemma above, we get the following formula for the capacity along $\GL(\dd)$-orbits.

\begin{corollary} \label{capa-char-formula-coro} Let $\V \in \rep(Q,\dd)$ and $A=(A(x))_{x \in Q_0} \in \GL(\dd)$. Then 
$$
\capa_Q(\V, \sigma)=(\chi_{\sigma}(A))^2 \cdot \capa_Q(A \cdot \V, \sigma).
$$
In particular, if $A \in \GL(\dd)_{\sigma}$ then $$\capa_Q(\V,\sigma)=\capa_Q(A \cdot \V, \sigma),$$ i.e. the capacity is constant along $\GL(\dd)_{\sigma}$-orbits.
\end{corollary}

We are now ready to define BL constants for bipartite quivers. This definition is inspired by Lieb's remarkable determinantal formula \cite{Lieb-1990} for the classical BL constants and Lemma \ref{cap-compute-lemma}. 

\begin{definition} (\textbf{Brascamp-Lieb constants for quiver datum}) \label{BL-defn-quivers}  Let $\V \in \rep(Q,\dd)$ be a $\dd$-dimensional representation and let $\tup=(p_1, \ldots, p_m) \in \QQ^m_{>0}$ be an $m$-tuple of positive rational numbers such that $\sum_{i=1}^n \dd(v_i)=\sum_{j=1}^m p_j \dd(w_j)$. We define the \emph{Brascamp-Lieb constant} $\bl_Q(\V,\tup)$ associated to $(\V,\tup)$ by
\begin{equation}
\bl_Q(\V,\tup)= \sup \left \{ \left( { \prod_{j=1}^m \det(Y_j)^{p_j} \over \prod_{i=1}^n \det \left( \sum_{j=1}^m p_j \left(\sum_{a \in \ar_{ij}}\V(a)^T\cdot Y_j \cdot \V(a) \right) \right)} \right)^{1 \over 2} \;\middle|\; Y_j \in \s^{+}_{\dd(w_j)} \right \}
\end{equation}

\noindent
(When computing the supremum above, the convention is that ${1 \over 0}$ is $\infty$.)
\end{definition}

\begin{remark}\label{BL-constant-rmk} 
\begin{enumerate}[(a)]
\item Using Lieb's formula (see \cite{Lieb-1990}), we obtain that $\bl_{\mathcal{Q}_m}(\V, \tup)$, where $\mathcal{Q}_m$ is the $m$-subspace quiver, is precisely the classical BL constant $\bl(\mathbf{\V}, \tup)$.\\

\item Keep the same notation as in the definition above. Let $\sigma_{\tup}$ be the integral weight of $Q$ defined by $\sigma_{\tup}(x_i)=\omega$ for all $i \in [n]$ and $\sigma_{\tup}(w_j)=-\omega\cdot p_j$ for all $j \in [m]$. Then, according to Lemma \ref{cap-compute-lemma}, we have that
\begin{equation}
\bl_Q(\V, \tup)=\begin{cases}
{1 \over {\sqrt[2\omega]{\omega^{-N} \capa_Q(\V,\sigma_{\tup})}}} & \text{~if~} \capa_Q(\V,\sigma_{\tup})>0  \\
\infty & \text{~if~} \capa_Q(\V,\sigma_{\tup})=0 
\end{cases}
\end{equation}

\item According to Theorem \ref{cap-semi-stab-thm}, $\bl_Q(\V, \tup) <\infty$ if and only if $\sum_{i=1}^n \dd(v_i)=\sum_{j=1}^m p_j \dd(w_j)$ and 

\begin{equation} \label{semi-stab-ineq-2}
\sum_{i=1}^n \dim \V'(v_i) \leq \sum_{j=1}^m p_j \dim \left( \sum_{i=1}^n \sum_{a \in \ar_{ij}}\V(a)(\V'(v_i)) \right),
\end{equation}
for all subspaces $\V'(v_i) \leq \RR^{\dd(v_i)}$, $\forall i \in [n]$.  In particular, when $Q$ is the $m$-subspace quiver, this recovers Bennet-Carbery-Christ-Tao's finitness result from \cite[Theorem 1.13]{BenCarChrTao-2008}.
\end{enumerate}
\end{remark}

\section{Geometric quiver data} \label{geo-quiver-data-sec}
Let $Q=(Q_0,Q_1,t,h)$ be a bipartite quiver with set of source vertices $Q_0^{+}=\{v_1, \ldots, v_n\}$, and set of sink vertices $Q_0^{-}=\{w_1,\ldots, w_m\}$. Let $\ar_{i,j}$ be the set of all arrows from $v_i$ to $w_j$ for all $i \in [n]$ and $j \in [m]$.

Let $\dd \in \NN^{Q_0}$ be a dimension vector and let $\sigma \in \mathbb{H}(\dd) \cap \ZZ^{Q_0}$ be a weight orthogonal to $\dd$ such that $\sigma$ is positive on $Q_0^{+}$ and negative on $Q_0^-$. Recall that $\sigma_+(v_i)=\sigma(v_i)$, $\forall i \in \{1, \ldots,n\}$, and $\sigma_-(w_j)=-\sigma(w_j)$, $\forall j \in \{1, \ldots,m\}$.

Let $V \in \rep(Q,\dd)$ be a $\dd$-dimensional representation and $T_{\V,\sigma}$ the Brascamp-Lieb operator associated to $(\V,\sigma)$. Recall that $T_{\V,\sigma}$ is a \emph{doubly stochastic} operator if $T_{\V,\sigma}(\Id)=T^*_{\V,\sigma}(\Id)=\Id$ which is equivalent to
\begin{equation}\label{geom-eq-1}
\sum_{j=1}^m \sigma_{-}(w_j) \sum_{a \in \ar_{i,j}} \V(a)^T \cdot \V(a)=\Id_{\dd(v_i)}, \forall i \in [n],
\end{equation}
and
\begin{equation}\label{geom-eq-2}
\sum_{i=1}^n \sigma_+(v_i) \sum_{a \in \ar_{i,j}} \V(a)\cdot \V(a)^T=\Id_{\dd(w_j)}, \forall j \in [m].
\end{equation}

\begin{definition} \label{quiver-geom-def} We call $(V,\sigma)$ a \emph{geometric quiver datum} if $V$ satisfies the matrix equations $(\ref{geom-eq-1})$ and $(\ref{geom-eq-2})$. 
\end{definition}

One of the advantages of working with geometric quiver data is that their capacity is known to be one (see \cite[Proposition 2.8 and Lemma 3.4]{GarGurOliWig-2015}), i.e. for a geometric datum $(\V,\sigma)$, we have that $$\capa_Q(\V,\sigma)=1.$$ 

\subsection{A character formula for the capacity of quiver data} Our goal in this section is to understand the matrix equations $(\ref{geom-eq-1})$ and $(\ref{geom-eq-2})$ in the context of quiver invariant theory. This will lead us to a character formula for the capacity of quiver data.

Recall that the affine space $\rep(Q,\dd)$ of $\dd$-dimensional representations of $Q$ is acted upon by the change-of-base group $\GL(\dd)=\prod_{x \in Q_0}\GL(\dd(x), \RR)$ by simultaneous conjugation. The character induced by $\sigma$ is denoted by $\chi_{\sigma}:\GL(\dd) \to \RR^{\times}$ and its kernel is denoted by $\GL(\dd)_{\sigma}$. 

\begin{theorem} \label{quiver-geom-data-thm}  
\begin{enumerate}[(i)]
\item For a $\sigma$-semi-stable representation $\V \in \rep(Q,\dd)$, consider the real algebraic variety
$$
\G_{\sigma}(\V):=\{A \in \GL(\dd) \mid (A\cdot \V, \sigma)\text{~is a geometric quiver datum}\}.
$$ 
Then
$$
\G_{\sigma}(\V) \neq \emptyset \Longleftrightarrow \V \text{~is~}\sigma-\text{polystable}.
$$

\item For a $\sigma$-semi-stable representation $\V \in \rep(Q,\dd)$, there exists a $\sigma$-polystable representation $\widetilde{\V}$ such that $\widetilde{\V} \in \overline{\GL(\dd)_{\sigma}\V}$. Furthermore, for any such $\widetilde{\V}$, the following formula holds:
\begin{equation}\label{capa-formula-eqn}
\capa_Q(\V,\sigma)=\capa_Q(\widetilde{\V},\sigma)=\chi_{\sigma}(A)^2, \forall A \in \G_{\sigma}(\widetilde{\V}).
\end{equation}
\end{enumerate}
\end{theorem}

To prove this theorem, we require the following important result. It has been proved by King \cite{K} over the field of complex numbers. Here, we explain how to prove it over the real numbers.

\begin{prop} \label{main-prop-KN} Let $Q$ be a bipartite quiver, $\dd \in \NN^{Q_0}$ a dimension vector, and $\sigma \in \ZZ^{Q_0}$ an integral weight of $Q$ such that $\sigma \cdot \dd=0$. Consider the action of $\GL(\dd)$ on $\rep(Q,\dd) \times \RR$ given by 
$$
A\cdot (\V,z)=(A\cdot \V, \chi_{\sigma}(A)z), \forall A \in \GL(\dd), (\V,z) \in \rep(Q,\dd)\times \RR.
$$  
For a $\sigma$-semi-stable representation $W \in \rep(Q,\dd)$, the following statements are equivalent:
\begin{enumerate}[(1)]
\item $W$ is $\sigma$-polystable;

\item the $\GL(\dd)_{\sigma}$-orbit of $W$ is closed in $\rep(Q,\dd)$;

\item the $\GL(\dd)$-orbit of $(W,1)$ is closed in $\rep(Q,\dd)\times \RR$;

\item there exists a representation $W' \in \GL(\dd)W$ such that
\begin{equation} \label{minimal-eqn-1-lemma}
\sum_{j=1}^m \sum_{a \in \ar_{i,j}} W'(a)^T \cdot W'(a)=\sigma_{+}(v_i)\Id_{\dd(v_i)}, \forall i \in [n],
\end{equation}
\begin{equation}\label{minimal-eqn-2-lemma}
\sum_{i=1}^n  \sum_{a \in \ar_{i,j}} W'(a)\cdot W'(a)^T=\sigma_{-}(w_j)\Id_{\dd(w_j)}, \forall j \in [m].
\end{equation}
\end{enumerate}
\end{prop}

\begin{proof} Over the field of complex numbers, the equivalence of $(2)$ and $(3)$ follows from King's work in \cite{K}. We also know that for rational representations of reductive groups (defined over $\RR$), the orbit of a point (defined over $\RR$) is closed over $\RR$ if and only if it is closed over $\CC$. This is a general result due to Birkes \cite[Corollary 5.3]{Bir-71}, and Borel and Harish-Chandra \cite[Proposition 2.3]{BorHarCha-62}. Consequently, we get the equivalence $(2) \Longleftrightarrow (3)$.

Next, let us prove that $(1) \Longrightarrow (3)$. For a $\sigma$-semi-stable representation $W$ if $W$ is $\sigma$-polystable then so is $W_{\CC}$ where $W_{\CC}$ is the base change of $W$ to $\CC$  (see \cite[Proposition 2.4 and Remark 2.5]{HosSch2017}). The latter is equivalent to the orbit of $(W_{\CC},1)$ under $G:=\prod_{x \in Q_0}\GL(\dd(x), \CC)$ being closed in $X:=\prod_{a \in Q_1}\Mat_{\dd(ha)\times \dd(ta)}(\CC)\times \CC$ (see \cite{K}). As mentioned above, this is further equivalent to $(3)$. Thus we get that $(1)$ implies $(3)$.

Now, let us check that $(2) \Longrightarrow (1)$. Since $W$ is $\sigma$-semi-stable, there exists a $\sigma$-polystable representation lying in the closure of $\GL(\alpha)_{\sigma}W$. Indeed, such a polystable representation can be taken to be the associated graded representation corresponding to a Jordan-H{\"o}lder filtration of $W$ in the category $\rep(Q)^{ss}_{\sigma}$ (for more details, see also Theorem \ref{quiver-geom-data-thm}{(ii)}). So, assuming $(2)$, this $\sigma$-polystable representation belongs to $\GL(\dd)_{\sigma}W$; in particular, it is isomorphic to $W$, and hence $W$ is $\sigma$-polystable.

It remains to show that $(3) \Longleftrightarrow (4)$. For this, consider the above action of $\GL(\dd)$ on $\rep(Q,\dd)\times \mathbb{R}$ at the level of Lie algebras:
$$
A \cdot (V,z)=\left( (A(ha)\cdot V(a)-V(a)\cdot A(ta))_{a \in Q_1}, \left(\sum_{x \in Q_0}\sigma(x) \tr(A(x)) \right)\cdot z \right),
$$
for every $A=(A(x))_{x \in Q_0} \in \RR^{\dd \times \dd}:=\prod_{x \in Q_0} \RR^{\dd(x) \times \dd(x)}$, $\V=(\V(a))_{a \in Q_1} \in \rep(Q,\dd)$, and $z \in \RR$. We equip $\rep(Q,\dd)\times \mathbb{R}$ with the inner product $\langle -,- \rangle$ induced from the natural inner product $\langle Y,Z \rangle=\tr(Y \cdot Z^T)$ on each $\RR^{\dd(ha) \times \dd(ta)}$, $a \in Q_1$. In what follows, we say that $(\V,z) \in \rep(Q,\dd)\times \mathbb R$ is \emph{minimal} (or critical) if
$$
\langle A \cdot (\V,z), (\V,z) \rangle=0, \forall A \in \RR^{\dd \times \dd},
$$
which is equivalent to 
$$
\sum_{a \in Q_1, ta=x} \V(a)^T \V(a)-\sum_{a \in Q_1, ha=x} \V(a)\V(a)^T=\sigma(x)z^2 \Id_{\dd(x)}, \forall x \in Q_0.
$$

According to the Kempf-Ness theory of minimal vectors over $\mathbb R$ (see for example \cite[Theorem 3.28]{Wal-2017} or \cite[Theorem 1.1]{Boh-Laf-2017}), the $\GL(\dd)$-orbit of $(W,1)$ is closed in $\rep(Q,\dd)\times \mathbb{R}$ if and only if there exists $A \in \GL(\dd)$ such that $A\cdot (W,1)=(A \cdot W, \chi_{\sigma}(A))$ is minimal. It is immediate to see that for an $A \in \GL(\dd)$, $(A \cdot W, \chi_{\sigma}(A))$ is minimal if and only if $W'$ satisfies the matrix equations $(\ref{minimal-eqn-1-lemma})$ and $(\ref{minimal-eqn-2-lemma})$ where $W'=A'\cdot W \in \GL(\dd)W$ and $A'=(|\chi_{\sigma}(A)|^{-{1 \over 2}} \cdot A(x))_{x \in Q_0} \in \GL(\dd)$. 
\end{proof}

\begin{remark} \label{KN-arbitrary-rmk} We point out that Proposition \ref{main-prop-KN} holds for arbitrary quivers. Specifically, let $Q$ be a connected quiver (not necessarily bipartite), $\dd \in \ZZ^{Q_0}_{\geq 0}$ a dimension vector, and $\sigma \in \ZZ^{Q_0}$ an integral weight such that $\sigma \cdot \dd=0$. Let $W \in \rep(Q, \dd)$ be a $\sigma$-semi-stable representation of $Q$. Then the proof above simply shows that the following statements are equivalent:
\begin{enumerate}
\item $W$ is $\sigma$-polystable;

\item the $\GL(\dd)_{\sigma}$-orbit of $W$ is closed in $\rep(Q,\dd)$;

\item the $\GL(\dd)$-orbit of $(W,1)$ is closed in $\rep(Q,\dd)\times \RR$;

\item there exists a representation $W' \in \GL(\dd)W$ such that
$$
\sum_{a \in Q_1, ta=x} W'(a)^T\cdot W'(a)-\sum_{a \in Q_1, ha=x} W'(a) \cdot W'(a)^T=\sigma(x) \Id_{\dd(x)}, \forall x \in Q_0.
$$
\end{enumerate}
\end{remark}

We are now ready to prove Theorem \ref{quiver-geom-data-thm}.

\begin{proof}[Proof of Theorem \ref{quiver-geom-data-thm}] 
$(i)$ Let $\V \in \rep(Q,\dd)$ be a $\sigma$-semi-stable representation. Define $W \in \rep(Q,\dd)$ by $W(a):=\sqrt{\sigma_+(ta)\sigma_{-}(ha)}\cdot  \V(a)$ for every $a \in Q_1$. Furthermore, we can now see that $\G_{\sigma}(\V) \neq \emptyset$ if and only if there exists an $A \in \GL(\dd)$ such that $A\cdot W$ satisfies $(\ref{minimal-eqn-1-lemma})$ and $(\ref{minimal-eqn-2-lemma})$. Via Proposition \ref{main-prop-KN}, this is further equivalent to $W$, and hence $\V$, being $\sigma$-polystable.

\bigskip
\noindent
$(ii)$ Since $\V$ is $\sigma$-semi-stable, $V$ has a Jordan-H{\"o}lder filtration in $\rep(Q)^{ss}_{\sigma}$.  After choosing a basis for each $V(x)=\RR^{\dd(x)}$ compatible with this filtration, we can construct a $1$-psg $\lambda' \in X_{*}(\GL(\dd)_{\sigma})$ and $h \in \GL(\dd)$ such that $\lim_{t \to 0} \lambda'(t) (h \cdot V)$ exists and is isomorphic to the direct sum of the composition factors of the chosen Jordan-H{\"o}lder filtration; in particular, the limit is $\sigma$-polystable.

Setting $\lambda(t)=h^{-1} \lambda'(t) h, \forall t \in \RR$, we get that $\lambda \in X_{*}(\GL(\dd)_{\sigma})$ and $\widetilde{\V}:=\lim_{t \to 0} \lambda(t)\cdot V$ exists and is $\sigma$-polystable. It is clear that $\widetilde{\V}$ belongs to the closure of $\GL(\dd)_{\sigma} \V$.

Finally, we know that the capacity $\capa_{Q}(-, \sigma)$ is continuous (see \cite[Section 7]{GarGurOliWig-2017}) and $\GL(\dd)_{\sigma}$-invariant by Corollary \ref{capa-char-formula-coro}. Consequently, for any $\sigma$-polystable representation $\widetilde{\V} \in \overline{\GL(\dd)_{\sigma} \V}$, we get that
$$
\capa_{Q}(\V, \sigma)=\capa_{Q}(\widetilde{\V}, \sigma)=\chi_{\sigma}(A)^2 \cdot \capa_Q(A\cdot \widetilde{\V}, \sigma)=\chi_{\sigma}(A)^2,
$$
for any $A \in \G_{\sigma}(\widetilde{\V})$.
\end{proof}

We use Theorem \ref{quiver-geom-data-thm} in an essential way  to prove the following factorization of the capacity of quiver representations. 

\begin{theorem}\label{mult-formula-capa} Let $\V \in \rep(Q,\dd)$ be a representation such that 
$$
\V(a)=\left(
\begin{matrix}
\V_1(a) & X(a)\\
0&\V_2(a)
\end{matrix}\right),
\forall a \in Q_1,
$$
where $\V_i \in \rep(Q,\dd_i)$, $i \in \{1,2\}$, are representations of $Q$, and $X(a) \in \RR^{\dd_1(ha)\times \dd_2(ta)}, \forall a \in Q_1$. If $\sigma \cdot  \ddim \V_1=0$ then
$$
\capa_Q(\V,\sigma)=\capa_Q(\V_1,\sigma)\cdot \capa_Q(\V_2,\sigma).
$$
\end{theorem}

\begin{proof} Let us consider the representation $\widetilde{\V} \in \rep(Q,\dd)$ given by
$$
\widetilde{\V}(a)=\left(
\begin{matrix}
\V_1(a) & 0\\
0&\V_2(a)
\end{matrix}\right),
\forall a \in Q_1.
$$
We claim that $\capa_Q(\V,\sigma)=\capa_Q(\widetilde{\V},\sigma)$. Indeed, for each $t \in \RR^*$, define
$$
\lambda(t)(x)=\left(
\begin{matrix}
t\Id_{\dd_1(x)} & 0\\
0&\Id_{\dd_2(x)}
\end{matrix}\right),
\forall x \in Q_0.
$$
Then, $(\lambda(t)\cdot \V)(a)=
\left(
\begin{matrix}
\V_1(a) & tX(a)\\
0&\V_2(a)
\end{matrix}\right)
, \forall a \in Q_1$, and so $\lim_{t \to 0} \lambda(t)\V=\widetilde{\V}$. We also have that $\chi_{\sigma}(\lambda(t))=t^{\sigma \cdot \dd_1}=1, \forall t \in \RR^*$, i.e. $\lambda \in X_{*}(\GL(\dd)_{\sigma})$. Using the fact that $\capa_Q(-,\sigma)$ is continuous and $\GL(\dd)_{\sigma}$-invariant (see \cite[Section 7]{GarGurOliWig-2017}  and Corollary \ref{capa-char-formula-coro}), we get that
$$
\capa_Q(\widetilde{\V},\sigma)=\lim_{t \to 0} \capa_Q(\lambda(t)\V, \sigma)=\capa_Q(\V,\sigma).
$$

In what follows, we show that
\begin{equation}\label{eqn-2-mult-formula}
\capa_Q(\widetilde{\V},\sigma)=\capa_Q(\V_1,\sigma)\cdot \capa_Q(V_2,\sigma),
\end{equation}
which will prove the desired factorization formula.

If $\capa_Q(\V,\sigma)=0$ then $\V$ is not $\sigma$-semi-stable by Theorem \ref{cap-semi-stab-thm}. In this case, we get that either $\V_1$ or $\V_2$ is not $\sigma$-semi-stable. This follows from the short exact sequence $0 \to \V_1 \to \V \to \V_2 \to 0$ of representations, and the fact that the category of $\sigma$-semi-stable representations of $Q$ is closed under extensions. Using Theorem \ref{cap-semi-stab-thm} again, this is equivalent to $\capa_Q(\V_1,\sigma)\cdot \capa_Q(\V_2,\sigma)=0$, proving $(\ref{eqn-2-mult-formula})$ when $\capa_Q(\V,\sigma)=0$.

Now, let us assume that $\capa_Q(\V,\sigma)>0$. In this case, we know from Theorem \ref{quiver-geom-data-thm} that there exists a $\sigma$-polystable representation $\V'_i \in \overline{\GL(\dd_i)_{\sigma}V_i}$ and a group element $A_i\in \GL(\dd_i)$ such that $(A_i\cdot \V'_i,\sigma)$ is a geometric quiver datum and 
$$
\capa_Q(V_i,\sigma)=\chi_{\sigma}(A_i)^2, \forall i \in \{1,2\}.
$$
In fact, we can choose each $\V'_i$ to be a degeneration of $\V_i$ along a $1$-psg of $\GL(\dd_i)_{\sigma}$. For $\widetilde{\V}':=\V'_1 \oplus \V'_2 \in \rep(Q,\dd)$ and $A:=A_1\oplus A_2 \in \GL(\dd)$, it is clear that $\widetilde{\V}' \in \overline{\GL(\dd)_{\sigma}\widetilde{V}}$ and $(A\cdot \widetilde{\V}',\sigma)$ is a geometric quiver datum. Consequently, we obtain from Theorem \ref{quiver-geom-data-thm}{(2)} that
$$
\capa_Q(\widetilde{V},\sigma)=\chi_{\sigma}(A)^2=\chi_{\sigma}(A_1)^2\cdot \chi_{\sigma}(A_2)^2=\capa_Q(V_1,\sigma)\cdot \capa_Q(V_2,\sigma).
$$
\end{proof}

\begin{remark} 
Let $\V \in \rep(Q,\dd)$ be a representation such that along every arrow, $\V$ is an upper triangular block matrix whose block entries are given by representations $\V_1, \ldots, \V_n$. Then Theorem \ref{mult-formula-capa} implies that
$$
\capa_Q(\V,\sigma)=\prod_{i=1}^n \capa_Q(\V_i,\sigma).
$$
\end{remark}

\subsection{Extremisable quiver data} For a representation $\V \in \rep(Q,\dd)$ and an $m$-tuple $Y=(Y_1,\ldots, Y_m)$ with $Y_j \in \s^{+}_{\dd(w_j)}$, $j \in [m]$, we set
$$
\capa_Q(\V, \sigma; Y):={ \prod_{i=1}^n \det \left( \sum_{j=1}^m \sigma_{-}(w_j)\left( \sum_{a \in \ar_{i,j}} \V(a)^T \cdot Y_j \cdot \V(a)  \right) \right)^{\sigma_{+}(v_i)}   \over \prod_{j=1}^m \det(Y_j)^{\sigma_{-}(w_j)}}
$$

\begin{definition} Let $\V \in \rep(Q,\dd)$ be such that $\capa_Q(\V, \sigma)>0$. We say that $(\V, \sigma)$ is \emph{gaussian-extremisable} if there exists an $m$-tuple $Y=(Y_j)_{j=1}^m$ with $Y_j \in \s^{+}_{\dd(w_j)}$, $j \in [m]$, such that 
$$
\capa_Q(\V,\sigma)=\capa_Q(\V,\sigma; Y)
$$

\noindent
We call any such tuple $Y$ a \emph{gaussian extremiser} for $(\V,\sigma)$.
\end{definition}

\begin{remark} Let $(\V,\sigma)$ be a gaussian-extremisable quiver datum with gaussian extremiser $Y=(Y_j)_{j=1}^m$. We claim that for any $A=(A(x))_{x \in Q_0} \in \GL(\dd)$, $(A \cdot \V, \sigma)$ is gaussian-extremisable with gaussian extremiser
\begin{equation}\label{eqn-g-extremizers}
\widetilde{Y}:=((A(w_j)^T)^{-1}\cdot Y_j \cdot A(w_j)^{-1} )_{j \in [m]}.
\end{equation}
Indeed, it is straightforward to see that 
$$\capa_Q(\V,\sigma)=\capa_Q(V,\sigma; Y)=(\chi_{\sigma}(A))^2\cdot \capa_Q(A\cdot \V,\sigma;\widetilde{Y}).$$ 
Using Corollary \ref{capa-char-formula-coro}, we then get that 
$$\capa_Q(A \cdot \V,\sigma)=(\chi_{\sigma}(A))^{-2}\cdot \capa_Q(\V, \sigma)=\capa_Q(A \cdot \V,\sigma;\widetilde{Y}),$$
and this proves our claim.
\end{remark}

Our next result gives necessary and sufficient conditions for $(\V, \sigma)$ to be gaussian-extremisable and explain how to construct all the gaussian extremisers from $\G_{\sigma}(\V)$.

\begin{theorem}\label{gaussian-extremisers-thm} Let $(\V, \sigma)$ be a quiver datum with $\V \in \rep(Q,\dd)$ and $\capa_Q(\V,\sigma)>0$. Then $\V$ is $\sigma$-polystable if and only if $(\V, \sigma)$ is gaussian-extremisable. If this is the case, the gaussian extremisers of $(\V, \sigma)$ are the $m$-tuples 
$$
(A(w_j)^T \cdot A(w_j))_{j \in [m]} \text{~with~}A \in \G_{\sigma}(\V). 
$$
\end{theorem}

\begin{proof} $(\Longrightarrow)$ Let $(\V, \sigma)$ be a quiver datum such that $\V$ is $\sigma$-polystable, and pick an arbitrary $A \in \G_{\sigma}(\V)$. Then $(A \cdot \V,\sigma)$ is geometric, and hence it is gaussian-extremisable with gaussian extremiser $(\Id_{\dd(w_j)})_{j=1}^m$. This observation combined with $(\ref{eqn-g-extremizers})$ shows that $(\V,\sigma)$ is gaussian-extremisable with gaussian extremiser $(A(w_j)^T \cdot A(w_j))_{j \in [m]}$.  

\smallskip
\noindent
$(\Longleftarrow)$ Let us assume now that $(\V, \sigma)$ is gaussian-extremisable and let $Y=(Y_j)_{j=1}^m$ be a gaussian extremiser for $(\V, \sigma)$. For each $i \in [n]$, let 
$$
M_i=\sum_{j \in [m]} \sigma_{-}(w_j) \left(\sum_{a \in \ar_{i,j}} \V(a)^{T} \cdot Y_j \cdot \V(a)  \right) \in \RR^{\dd(v_i)\times \dd(v_i)}.
$$

Since $\capa_Q(\V, \sigma)$ is assumed to be non-zero and 
$$ 
\capa_Q(\V,\sigma)={\prod_{i \in [n]} \det(M_i)^{\sigma_{+}(v_i)}  \over \prod_{j \in [m]} \det(Y_j)^{\sigma_{-}(w_j)}},$$
we conclude that each $M_i$ is a positive definite matrix. Define
$$
A(v_i)=M_i^{1 \over 2}, \forall i \in [n], \text{~and~}A(w_j)=Y_j^{1 \over 2}, \forall j \in [m].
$$

\noindent
\textbf{Claim:} If $A:=(A(v_i), A(w_j))_{i \in [n], j \in [m]}$ then $A \in \G_{\sigma}(\V)$.

\begin{proof}[Proof of \textbf{Claim}] 
 We begin by checking that $A \cdot \V$ satisfies equation  $(\ref{geom-eq-1})$. We have that 
\begin{align*}
&\sum_{j \in [m]} \sigma_{-}(w_j) \left( \sum_{a \in \ar_{i,j}} (A \cdot \V)^T(a) \cdot (A \cdot \V)(a) \right)=\\
=&\sum_{j \in [m]} \sigma_{-}(w_j) \left( \sum_{a \in \ar_{i,j}} A(ta)^{-T}\cdot \V(a)^T \cdot A(ha)^{T}\cdot A(ha) \cdot \V(a) \cdot  A(ta)^{-1} \right)\\
=&\sum_{j \in [m]} \sigma_{-}(w_j) \left( \sum_{a \in \ar_{i,j}} M_i^{-{1 \over 2}}\cdot \V(a)^T \cdot Y_j^{{1 \over 2}} \cdot Y_j^{{1 \over 2}} \cdot \V(a) \cdot  M_i^{-{1 \over 2}} \right)\\
=&M_i^{-{1 \over 2}}\cdot \left( \sum_{j \in [m]} \sigma_{-}(w_j) \left( \sum_{a \in \ar_{i,j}}  \V(a)^T \cdot Y_j \cdot \V(a) \right) \right) \cdot  M_i^{-{1 \over 2}} \\
=&M_i^{-{1 \over 2}}\cdot M_i \cdot  M_i^{-{1 \over 2}} =\Id_{\dd(v_i)}, \forall i \in [n],
\end{align*}
i.e. $A \cdot \V$ satisfies equation $(\ref{geom-eq-1})$. To show that $A \in \G_{\sigma}(\V)$, it remains to check that $A \cdot \V$ satisfies equation $(\ref{geom-eq-2})$, as well. For this, we first show that
\begin{equation}\label{g-extremisers-eqn}
\sum_{i \in [n]} \sigma_{+}(v_i) \left( \sum_{a \in \ar_{i,j}} \V(a)\cdot M_i^{-1}\cdot \V(a)^T \right)=Y_j^{-1}, \forall j \in [m].
\end{equation}

To prove $(\ref{g-extremisers-eqn})$, we adapt the proof strategy for the implication $(b) \Longrightarrow (c)$ in \cite[Proposition 3.6]{BenCarChrTao-2008} to our more general quiver set-up. After taking logarithms in the formula for the capacity in Lemma \ref{cap-compute-lemma}, we get that $Y$ is a minimiser for the quantity
\begin{equation} \label{eqn-proof-g-min}
\sum_{i \in [n]}\sigma_{+}(v_i)\log \left( \det \left( \sum_{j \in [m]} \sigma_{-}(w_j) \left( \sum_{a \in \ar_{i,j}} \V(a)^T\cdot Y_j \cdot \V(a)\right) \right) \right)-\sum_{j \in [m]} \sigma_{-}(w_j) \log (\det(Y_j))
\end{equation}
Now, let us fix $j \in [m]$ and let $Q_j \in \RR^{\dd(w_j) \times \dd(w_j)}$ be an arbitrary symmetric matrix. Replacing $Y_j$ by $Y_j+\epsilon Q_j$ in $(\ref{eqn-proof-g-min})$, we can see that $\epsilon=0$ is a minimiser for 
$$
A(\epsilon)-B(\epsilon),
$$
where
\begin{align*}
&A(\epsilon)=\sum_{i \in [n]}\sigma_{+}(v_i)\log \left( \det \left( M_i+\epsilon \sigma_{-}(w_j) \left( \sum_{a \in \ar_{i,j}} \V(a)^T\cdot  Q_j \cdot \V(a) \right) \right) \right)\\
=&\sum_{i \in [n]}\sigma_{+}(v_i)\log(\det M_i)+\sum_{i \in [n]}\sigma_{+}(v_i)\log \left( \det \left( \Id_{\dd(v_i)}+\epsilon \sigma_{-}(w_j)M_i^{-1} \left( \sum_{a \in \ar_{i,j}} \V(a)^T\cdot  Q_j \cdot \V(a) \right) \right) \right) 
\end{align*}
and
\begin{align*}
&B(\epsilon)=\sum_{j' \in [m], j' \neq j} \sigma_{-}(w_{j'}) \log (\det(Y_{j'}))+\sigma_{-}(w_j) \log (\det(Y_j+\epsilon Q_j))\\
=&\sum_{j' \in [m]} \sigma_{-}(w_{j'}) \log (\det(Y_{j'}))+\sigma_{-}(w_j) \log (\det(\Id_{\dd(w_j)}+\epsilon Y_j^{-1}\cdot Q_j))
\end{align*}
Since ${d \over d\epsilon} \log (\det(\Id+\epsilon Q))|_{\epsilon=0}=\tr(Q)$ for any symmetric matrix $Q$, we get that
\begin{align*}
0=&{d \over d\epsilon}(A(\epsilon)-B(\epsilon))|_{\epsilon=0}\\
=&\sum_{i \in [n]}\sigma_{+}(v_i)\tr\left(\sigma_{-}(w_j)M_i^{-1} \left( \sum_{a \in \ar_{i,j}} \V(a)^T\cdot  Q_j \cdot \V(a) \right)\right)-\sigma_{-}(w_j) \tr(Y_j^{-1}\cdot Q_j)\\
=&\sigma_{-}(w_j) \left(  \sum_{i \in [n]}\tr\left(\sigma_{+}(v_i)M_i^{-1} \left( \sum_{a \in \ar_{i,j}} \V(a)^T\cdot  Q_j \cdot \V(a) \right)\right)-\tr(Y_j^{-1}\cdot Q_j)   \right). 
\end{align*}
Rearranging the factors inside the trace, we obtain that  
\begin{equation} \label{tr-formula-g-min}
\tr \left( Q_j \cdot \left( \sum_{i \in [n]}\sigma_{+}(v_i) \left( \sum_{a \in \ar_{i,j}} \V(a)\cdot M_i^{-1}\cdot \V(a)^T \right)-Y_j^{-1} \right) \right)=0,
\end{equation}
holds for all symmetric matrices $Q_j$. Since $\sum_{i \in [n]}\sigma_{+}(v_i) \left( \sum_{a \in \ar_{i,j}} \V(a)\cdot M_i^{-1}\cdot \V(a)^T \right)-Y_j^{-1}$ is a symmetric matrix, we can see that $(\ref{tr-formula-g-min})$ yields the desired formula $(\ref{g-extremisers-eqn})$. 

Finally, for each $j \in [m]$, we get via $(\ref{g-extremisers-eqn})$ that 
\begin{align*}
&\sum_{i \in [n]} \sigma_{+}(v_i) \left(  \sum_{a \in \ar_{i,j}} (A\cdot \V)(a) \cdot (A\cdot V)^T(a) \right)=\\
=&A(w_j)\cdot \left(\sum_{i \in [n]} \sigma_{+}(v_i) \left(  \sum_{a \in \ar_{i,j}} \V(a)\cdot A(v_i)^{-1}\cdot A(v_i)^{-T} \cdot V^T(a) \right) \right) \cdot A(w_j)^T\\
=&A(w_j)\cdot \left(\sum_{i \in [n]} \sigma_{+}(v_i) \left(  \sum_{a \in \ar_{i,j}} \V(a)\cdot M_i^{-1} \cdot V^T(a) \right) \right) \cdot A(w_j)^T\\
=&A(w_j)\cdot Y_j^{-1} \cdot A(w_j)^T=\Id_{\dd(w_j)},
\end{align*}
i.e. $A \cdot \V$ satisfies equation $(\ref{geom-eq-2})$, as well. This finishes the proof of our claim.
\end{proof}

It now follows from Theorem \ref{quiver-geom-data-thm} and the claim above that $\V$ is indeed $\sigma$-polystable, and the gaussian extremiser $Y$ is of the form $(A(w_j)^T \cdot A(w_j))_{j \in [m]}$ with $A \in \G_{\sigma}(\V)$.
\end{proof}

Finally, we give necessary and sufficient conditions for a quiver datum to have unique gaussian extremisers. 

\begin{theorem} \label{uniqueness-g-extremals-thm} Let $(\V, \sigma)$ be a quiver datum with $\V \in \rep(Q,\dd)$ and $\capa_Q(\V,\sigma)>0$. If $\V$ is $\sigma$-stable and $\End_Q(V)=\RR$ then $(\V, \sigma)$ has unique gaussian extremisers (up to scaling). Conversely, if $(\V,\sigma)$ has unique gaussian extremisers (up to scaling) then $\V$ is $\sigma$-stable. 
\end{theorem}

\begin{proof} According to Kempf-Ness theory (see for example \cite[Theorem 1.1(i)]{Boh-Laf-2017}), if $\V$ is $\sigma$-polystable then for any two $A_1, A_2 \in \G_{\sigma}(\V)$, we have that
$$
A_2 \cdot \V \in \mathbf{O}(\dd)\cdot(A_1 \cdot \V),
$$
where $\mathbf{O}(\dd)$ denotes the subgroup of $\GL(\dd)$ consisting of all tuples of orthogonal matrices. In other words,
\begin{equation} \label{orbit-g-extremals-eqn}
A_2^{-1}\cdot h \cdot A_1 \in \Stab_{\GL(\dd)}(\V)=\End_Q(\V)^{\times} \text{~for some~} h \in \mathbf{O}(\dd).
\end{equation}

\smallskip
\noindent
To prove the first claim, let us assume that $\V$ is $\sigma$-stable and $\End_Q(\V)=\{(\lambda\Id_{\dd(x)})_{x \in Q_0} \mid \lambda \in \RR \}$. We know from Theorem \ref{gaussian-extremisers-thm} that $(\V,\sigma)$ is gaussian-extremisable, and let $A_1, A_2 \in \G_{\sigma}(\V)$. Then, by $(\ref{orbit-g-extremals-eqn})$, we can write
$$
A_2=\lambda (h \cdot A_1)
$$ 
for some $h \in \mathbf{O}(\dd)$ and $\lambda \in \RR^{\times}$. Consequently, we get that
$$
A_2(w_j)^T \cdot A_2(w_j)=\lambda^2 (A_1(w_j)^T \cdot A_1(w_j)), \forall j \in [m].
$$
It now follows from Theorem \ref{gaussian-extremisers-thm} that $(\V, \sigma)$ has unique gaussian extremisers, up to scaling.

For the second claim, let us assume $(V, \sigma)$ has unique gaussian-extremisers. By Theorem \ref{gaussian-extremisers-thm}, we know that $\V$ is $\sigma$-polystable. Let us assume for a contradiction that $\V=\V_1\oplus \V_2$ with $\V_1 \in \rep(Q,\dd_1)$, $\V_1 \in \rep(Q,\dd_2)$, two proper $\sigma$-polystable subrepresentations of $\V$. So we can choose $A_1 \in \G_{\sigma}(\V_1)$ and $A_2 \in \G_{\sigma}(\V_2)$ by Theorem \ref{gaussian-extremisers-thm}. Then, for any two scalars $\lambda_1, \lambda_2 \in \RR^{\times}$ with $|\lambda_1| \neq |\lambda_2|$, the gaussian extermisers for $(\V, \sigma)$ corresponding to
$$
\left(
\begin{matrix}
A_1&0\\
0&A_2
\end{matrix}
\right)
\text{~and~}
\left(
\begin{matrix}
\lambda_1 A_1&0\\
0&\lambda_2 A_2
\end{matrix}
\right)
$$
are not a scalar multiple of each other (contradiction). This finishes the proof. 
\end{proof}

\subsection{Proof of Theorem \ref{main-thm-2}} Parts $(1)$ and $(2)$ of Theorem \ref{main-thm-2} are proved in Theorem \ref{quiver-geom-data-thm}. Part $(3)$ is proved in Theorem \ref{mult-formula-capa}. Finally parts $(4)$ and $(5)$ are proved in Theorems \ref{gaussian-extremisers-thm} and \ref{uniqueness-g-extremals-thm}.

\section{Structural results for Brascamp-Lieb constants for quiver data}\label{gen-BL-const-summary-sec} 

In this section we explain how to rephrase our main results, Theorems \ref{cap-semi-stab-thm} and \ref{main-thm-2}, in terms of BL constants. Let $Q=(Q_0,Q_1,t,h)$ be a bipartite quiver with set of source vertices $Q_0^{+}=\{v_1, \ldots, v_n\}$, and set of sink vertices $Q_0^{-}=\{w_1,\ldots, w_m\}$. Denote by $\ar_{i,j}$ the set of all arrows from $v_i$ to $w_j$ for all $i \in [n]$ and $j \in [m]$.

Let $\dd$ be a dimension vector of $Q$ and $\tup=(p_1, \ldots, p_m)$ an $m$-tuple of positive rational numbers such that 
\begin{equation}\label{ortho-eqn}
\sum_{i=1}^n\dd(v_i)=\sum_{j=1}^m p_j\cdot \dd(w_j).
\end{equation}

A $\dd$-dimensional representation $\V$ of $Q$ is said to be \emph{$\tup$-semi-stable} if 
$$\sum_{i=1}^n \dim_{\RR} \V'(v_i) \leq \sum_{j=1}^m p_j \dim_{\RR}\V'(w_j),$$
for all subrepresentations $\V' \leq \V$. We say that $\V$ is \emph{$\tup$-stable} if the inequality above is strict for all proper subrepresentations $\V'$ of $\V$. A representation is said to be \emph{$\tup$-polystable} if it is a finite direct sum of $\tup$-stable representations. 

Let $\omega$ be the least common denominator of $p_1, \ldots, p_m$. Then the weight $\sigma_{\tup}$ of $Q$ induced by $\tup$ is defined by 
$$
\sigma_{\tup}(v_i)=\omega, \forall i \in [n], \text{~and~}\sigma_{\tup}(w_j)=-\omega \cdot p_j, \forall j \in [m].
$$ 
We also denote by $\chi_{\tup}$ the character of $\GL(\dd)$ induced by $\sigma_{\tup}$, i.e.
$$
\chi_{\tup}(A)=\prod_{i=1}^n \det(A(v_i))^{\omega}\cdot \prod_{j=1}^m \det(A(w_j))^{-wp_j}, \forall A \in \GL(\dd).
$$

Now, let $(\V, \tup)$ be a quiver datum with $\V \in \rep(Q,\dd)$. Then it is clear that
\begin{itemize}
\item $(\ref{ortho-eqn})$ is equivalent to $\sigma_{\tup}\cdot \dd=0$; and

\item $\V$ being $\tup$-semi-stable/stable/polystable is equivalent to \\ $\V$ being $\sigma_{\tup}$-semi-stable/stable/polystable.
\end{itemize}

Recall that the Brascamp-Lieb constant associated to the quiver datum $(\V,\tup)$ is 
\begin{equation}\label{bl-formula-defn}
\bl_Q(\V, \tup)=\sup \left \{ \left( { \prod_{j=1}^m \det(Y_j)^{p_j} \over \prod_{i=1}^n \det \left( \sum_{j=1}^m p_j \left(\sum_{a \in \ar_{ij}}\V(a)^T\cdot Y_j \cdot \V(a) \right) \right)} \right)^{1 \over 2}  \right \}, 
\end{equation}
where the supremum is taken over all positive definite matrices $Y_j \in \RR^{\dd(w_j)\times \dd(w_j)}$, $j \in [m]$. We say that $(\V, \tup)$ is \emph{feasible} if $\bl_Q(\V, \tup) < \infty$. A feasible quiver datum $(\V, \tup)$ is said to be \emph{gaussian-extremisable} if there exist positive definite matrices $Y_j \in \RR^{\dd(w_j)\times \dd(w_j)}$, $j \in [m]$, for which the supremum is attained in $(\ref{bl-formula-defn})$. If this is the case, we call such an $m$-tuple $(Y_1, \ldots,Y_m)$ a \emph{gaussian extremiser} for $(\V, \tup)$.

When working with BL constants, we ``scale'' the definition of a geometric quiver datum as follows: We say that $(\V,\tup)$ is a \emph{geometric BL quiver datum} if 
\begin{equation} \label{BL-geom-eq-3}
\sum_{j=1}^m p_j \sum_{a \in \ar_{i,j}} (\V(a))^T \cdot \V(a)=\Id_{\dd(v_i)}, \forall i \in [n],
\end{equation}
and
\begin{equation} \label{BL-geom-eq-4}
\sum_{i=1}^n  \sum_{a \in \ar_{i,j}} \V(a)\cdot (\V(a))^T=\Id_{\dd(w_j)}, \forall j \in [m].
\end{equation} 
(When $Q$ is the $m$-subspace quiver $\mathcal{Q}_m$, this is the definition of a geometric datum introduced in \cite[Section 2]{BenCarChrTao-2008}.)

Let us now explain the relationship between the algebraic varieties $\BL_{\tup}(\V)$ and $\G_{\sigma_{\tup}}(\V)$. For a representation $\V \in \rep(Q,\dd)$, consider the real algebraic variety
$$
\BL_{\tup}(\V)=\{A \in \GL(\dd) \mid (A\cdot \V,\tup) \text{~is a geometric BL quiver datum}\}.
$$
For $A \in \GL(\dd)$, define $\widetilde{A} \in \GL(\dd)$ by $\widetilde{A}(v_i)=A(v_i)$, $ \forall i \in [n]$, and $\widetilde{A}(w_j)=\omega^{-{1 \over 2}} \cdot A(w_j)$, $\forall j \in [m]$. Then it is straightforward to check that 
\begin{enumerate}[(i)]
\item $A \in \BL_{\tup}(\V) \Longleftrightarrow \widetilde{A} \in \G_{\sigma_{\tup}}(\V)$;\\

\item $\omega^N \cdot \chi_{\tup}(A)^2=\chi_{\sigma_{\tup}}(\widetilde{A})^2$;\\

\item for any $\tup$-polystable representation $\V \in \rep(Q,\dd)$, Theorem \ref{quiver-geom-data-thm}{(2)} and Remark \ref{BL-constant-rmk}{(2)} yield
$$
\bl_Q(\V, \tup)={1 \over {\sqrt[2\omega]{\omega^{-N}\cdot \capa_Q(\V,\sigma_{\tup})}}}={1 \over {\sqrt[2\omega]{\omega^{-N}\cdot \chi_{\sigma_{\tup}}(\widetilde{A})^2}}}={1 \over {\sqrt[2\omega]{\chi_{\tup}(A)^2}}}
$$
for any $A \in \BL_{\tup}(\V)$; and

\item an $m$-tuple $(Y_1, \ldots, Y_m)$ of positive definite matrices is a gaussian extremiser for $(\V, \tup)$ if and only if $(\omega^{-1} \cdot Y_1, \ldots, \omega^{-1} \cdot Y_m)$ is a gaussian extremiser for $(\V, \sigma_{\tup})$. 
\end{enumerate}

Consequently, applying Theorems \ref{cap-semi-stab-thm} and \ref{main-thm-2} to this set-up yields the following structural result on BL constants for arbitrary bipartite quivers.

\begin{theorem}\label{sumarry-BL-consts-thm} Keep the same notation as above.
\begin{enumerate}
\item (\textbf{Finitness of BL constants}) Let $(\V, \tup)$ be a quiver datum with $\V \in \rep(Q,\dd)$. Then $\bl_Q(\V, \tup) < \infty$ if and only if $\V$ is $\tup$-semi-stable. 

\bigskip

\item (\textbf{Kempf-Ness theorem for real quiver representations}) Let $(\V, \tup)$ be a feasible quiver datum with $\V \in \rep(Q,\dd)$ and consider the real algebraic variety
$$
\BL_{\tup}(\V):=\{A \in \GL(\dd) \mid (A\cdot \V, \tup)\text{~is a geometric BL quiver datum}\}.
$$ 
Then
$$
\BL_{\tup}(\V) \neq \emptyset \Longleftrightarrow \V \text{~is~}\tup-\text{polystable}.
$$

\bigskip

\item (\textbf{Character formula for BL constants}) Let $(\V, \tup)$ be a feasible quiver datum with $\V \in \rep(Q,\dd)$. Then there exists a $\tup$-polystable representation $\widetilde{\V}$ such that $\widetilde{\V} \in \overline{\GL(\dd)_{\tup}\V}$. Furthermore, for any such $\widetilde{\V}$, the following formula holds:
$$
\bl_Q(\V,\tup)=\bl_Q(\widetilde{\V},\mathbf{p})= |\chi_{\tup}(A)|^{-{1 \over \omega}}, \forall A \in \BL_{\tup}(\widetilde{\V}).
$$ 

\bigskip

\item (\textbf{Factorization of BL constants}) Let $\V \in \rep(Q,\dd)$ be a $\dd$-dimensional representation. Assume that
$$
\V(a)=\left(
\begin{matrix}
\V_1(a) & X(a)\\
0&\V_2(a)
\end{matrix}\right),
\forall a \in Q_1,
$$
where $\V_i \in \rep(Q,\dd_i)$, $i \in \{1,2\}$, are representations of $Q$, and $X(a) \in \RR^{\dd_1(ha)\times \dd_2(ta)}$ for every $a \in Q_1$. If $\tup$ and $\ddim_{\V_1}$ are orthogonal then
$$
\bl_Q(\V,\tup)=\bl_Q(\V_1,\tup)\cdot \bl_Q(\V_2,\tup).
$$

\bigskip

\item (\textbf{Gaussian extremisers: existence}) A feasible quiver datum $(\V, \tup)$ with $\V \in \rep(Q,\dd)$ is gaussian-extremisable if and only if $\V$ is $\tup$-polystable. If this is the case, the gaussian extremisers of $(\V, \tup)$ are the $m$-tuples of matrices
$$
(A(w_j)^T \cdot A(w_j))_{j \in [m]} \text{~with~}A \in \BL_{\mathbf{p}}(\V). 
$$

\item (\textbf{Gaussian extremisers: uniqueness}) If a feasible quiver datum $(\V, \tup)$ with $\V \in \rep(Q,\dd)$ has unique gaussian extremisers (up to scaling) then $\V$ is $\tup$-stable. Conversely, if $\V$ is $\tup$-stable and $\End_Q(V)=\RR$ then $(\V, \tup)$ has unique gaussian extremisers (up to scaling).
\end{enumerate}
\end{theorem}

\begin{rmk} 
\begin{enumerate}
\item In a sequel to the current work, we plan to further study the capacity and BL-constants associated to quiver data by focusing on the constructive, algorithmic aspects of the real algebraic varieties $\G_{\sigma}(\V)$ and $\BL_{\mathbf{p}}(\V)$ introduced in this paper.  

\item The main results of this paper have already found applications to Edmonds' Problem in algebraic complexity (see \cite{ChiKli-Edmonds-2020}) and simultaneous robust subspace recovery in machine learning (see \cite{ChiKli-SRSR-2020}). Further applications of Theorem  \ref{quiver-geom-data-thm} and Proposition \ref{main-prop-KN} to Radial Isotropy and Paulsen's Problem for matrix frames will appear in a future paper on the subject.
\end{enumerate}
\end{rmk}

\providecommand{\bysame}{\leavevmode\hbox to3em{\hrulefill}\thinspace}
\providecommand{\MR}{\relax\ifhmode\unskip\space\fi MR }
\providecommand{\MRhref}[2]{%
  \href{http://www.ams.org/mathscinet-getitem?mr=#1}{#2}
}
\providecommand{\href}[2]{#2}


\begin{thebibliography}{BCELM11}

\bibitem[Bal89]{Ball-1989}
K.~Ball, \emph{Volumes of sections of cubes and related problems}, Geometric
  aspects of functional analysis (1987--88), Lecture Notes in Math., vol. 1376,
  Springer, Berlin, 1989, pp.~251--260. \MR{1008726}

\bibitem[Bar98]{Bar-1998}
F.~Barthe, \emph{On a reverse form of the {B}rascamp-{L}ieb inequality},
  Invent. Math. \textbf{134} (1998), no.~2, 335--361. \MR{1650312}

\bibitem[BBFL18]{Ben-Bez-Flo-Lee-2018}
J.~Bennett, N.~Bez, T.~C. Flock, and S.~Lee, \emph{Stability of the
  {B}rascamp-{L}ieb constant and applications}, Amer. J. Math. \textbf{140}
  (2018), no.~2, 543--569. \MR{3783217}

\bibitem[BCCT08]{BenCarChrTao-2008}
J.~Bennett, A.~Carbery, M.~Christ, and T.~Tao, \emph{The {B}rascamp-{L}ieb
  inequalities: finiteness, structure and extremals}, Geom. Funct. Anal.
  \textbf{17} (2008), no.~5, 1343--1415. \MR{2377493}

\bibitem[BCELM11]{Bar-et-al-2011}
F.~Barthe, D.~Cordero-Erausquin, M.~Ledoux, and B.~Maurey, \emph{Correlation
  and {B}rascamp-{L}ieb inequalities for {M}arkov semigroups}, Int. Math. Res.
  Not. IMRN (2011), no.~10, 2177--2216. \MR{2806562}

\bibitem[BCT06]{BenCarTao-2006}
J.~Bennett, A.~Carbery, and T.~Tao, \emph{On the multilinear restriction and
  {K}akeya conjectures}, Acta Math. \textbf{196} (2006), no.~2, 261--302.
  \MR{2275834}

\bibitem[BHC62]{BorHarCha-62}
A.~Borel and Harish-Chandra, \emph{Arithmetic subgroups of algebraic groups},
  Ann. of Math. (2) \textbf{75} (1962), 485--535. \MR{0147566}

\bibitem[Bir71]{Bir-71}
D.~Birkes, \emph{Orbits of linear algebraic groups}, Ann. of Math. (2)
  \textbf{93} (1971), 459--475. \MR{0296077}

\bibitem[BL17]{Boh-Laf-2017}
C.~{B{\"o}hm} and R.~A. {Lafuente}, \emph{{Real geometric invariant theory}},
  ArXiv e-prints (2017).

\bibitem[BPR06]{BasPolRoy-2006}
S.~Basu, R.~Pollack, and M.-F. Roy, \emph{Algorithms in real algebraic
  geometry}, second ed., Algorithms and Computation in Mathematics, vol.~10,
  Springer-Verlag, Berlin, 2006. \MR{2248869}

\bibitem[CDP15]{CheDafPao-2015}
W.-K. Chen, N.~Dafnis, and G.~Paouris, \emph{Improved {H}\"{o}lder and reverse
  {H}\"{o}lder inequalities for {G}aussian random vectors}, Adv. Math.
  \textbf{280} (2015), 643--689. \MR{3350230}

\bibitem[CK20a]{ChiKli-Edmonds-2020}
C.~{Chindris} and D.~{Kline}, \emph{{Edmonds' problem and the membership
  problem for orbit semigroups of quiver representations}}, arXiv e-prints
  (2020), arXiv:2008.13648.

\bibitem[CK20b]{ChiKli-SRSR-2020}
\bysame, \emph{{Simultaneous robust subspace recovery and semi-stability of
  quiver representations}}, arXiv e-prints (2020), arXiv:2003.02962.

\bibitem[DGOS18]{DvirGarOliSol-2018}
Z.~Dvir, A.~Garg, R.~Oliveira, and J.~Solymosi, \emph{Rank bounds for design
  matrices with block entries and geometric applications}, Discrete Anal.
  (2018), Paper No. 5, 24. \MR{3775994}

\bibitem[DH16]{Dvir-Hi-2016}
Z.~Dvir and G.~Hu, \emph{Sylvester-{G}allai for arrangements of subspaces},
  Discrete Comput. Geom. \textbf{56} (2016), no.~4, 940--965. \MR{3561796}

\bibitem[DM17]{HarmVisu-2017}
H.~Derksen and V.~Makam, \emph{Polynomial degree bounds for matrix
  semi-invariants}, Adv. Math. \textbf{310} (2017), 44--63. \MR{3620684}

\bibitem[{Fra}18]{Franks-2018}
C.~{Franks}, \emph{{Operator scaling with specified marginals}}, ArXiv e-prints
  (2018).

\bibitem[GGOW15]{GarGurOliWig-2015}
A.~{Garg}, L.~{Gurvits}, R.~{Oliveira}, and A.~{Wigderson}, \emph{{Operator
  scaling: theory and applications}}, ArXiv e-prints (2015).

\bibitem[GGOW18]{GarGurOliWig-2017}
A.~Garg, L.~Gurvits, R.~Oliveira, and A.~Wigderson, \emph{Algorithmic and
  optimization aspects of {B}rascamp-{L}ieb inequalities, via operator
  scaling}, Geom. Funct. Anal. \textbf{28} (2018), no.~1, 100--145.
  \MR{3777414}

\bibitem[HS17]{HosSch2017}
V.~{Hoskins} and F.~{Schaffhauser}, \emph{{Rational points of quiver moduli
  spaces}}, ArXiv e-prints (2017).

\bibitem[Kin94]{K}
A.D. King, \emph{Moduli of representations of finite-dimensional algebras},
  Quart. J. Math. Oxford Ser.(2) \textbf{45} (1994), no.~180, 515--530.

\bibitem[Lie90]{Lieb-1990}
E.~H. Lieb, \emph{Gaussian kernels have only {G}aussian maximizers}, Invent.
  Math. \textbf{102} (1990), no.~1, 179--208. \MR{1069246}

\bibitem[Wal17]{Wal-2017}
N.~R. Wallach, \emph{Geometric invariant theory}, Universitext, Springer, Cham,
  2017, Over the real and complex numbers. \MR{3700428}

\end{thebibliography}
\end{document}